\documentclass[final,leqno]{siamltex704}
\usepackage{hyperref}
\usepackage{amsmath}
\usepackage{graphicx}
\usepackage{array}
\usepackage{lineno}
\usepackage[notcite,notref]{showkeys}
\usepackage{tikz}
\usepackage{amsfonts,amssymb}
\usepackage{dsfont}
\usepackage{pifont}
\usepackage{subfigure}

\newtheorem{remark}{Remark}[section]
\renewcommand{\ldots}{\dotsc}

\newtheorem{model-problem}{Problem}

\newcommand{\bn}{\textbf{n}}

\newcommand{\br}{\textbf{r}}
\newcommand{\be}{\textbf{e}}
\newcommand{\bv}{\textbf{v}}

\def\T{{\mathcal T}}

\def\F{{\mathcal F}}
\def\E{{\mathcal E}}
\def\W{{\mathcal W}}

\def\bn{{\bf n}}

\def\3bar{{|\hspace{-.02in}|\hspace{-.02in}|}}

\def\bbq{\begin{equation*}}
\def\eeq{\end{equation*}}
\def\br{\begin{eqnarray}}
\def\er{\end{eqnarray}}
\def\brr{\begin{eqnarray*}}
\def\err{\end{eqnarray*}}

\def\O{\Omega}

\def\E{{\mathcal E}}

\def\pa{\partial}

\def\bn{{\bf n}}

\newtheorem{WG}{Weak Galerkin Algorithm}

\begin{document}

\setlength{\parindent}{0.25in} \setlength{\parskip}{0.08in}

\title{High Order Morley Elements for Biharmonic Equations on Polytopal Partitions}

\author{
Dan Li\thanks{Jiangsu Key Laboratory for NSLSCS, School of Mathematical Sciences,  Nanjing Normal University, Nanjing 210023, China (danlimath@163.com). {The research of Dan Li was supported by Jiangsu Funding Program for Excellent
Postdoctoral Talent and} China National Natural Science Foundation Grant (No. 12071227).}
\and
Chunmei Wang \thanks{Department of Mathematics, University of Florida, Gainesville, FL 32611 (chunmei.wang@ufl.edu). The research of Chunmei Wang was partially supported by National Science Foundation Grants DMS-2136380 and DMS-2206332.}
\and
Junping Wang\thanks{Division of Mathematical Sciences, National Science Foundation, Alexandria, VA 22314 (jwang@nsf.gov). The research of Junping Wang was supported by the NSF IR/D program, while working at National Science Foundation. However, any opinion, finding, and conclusions or recommendations expressed in this material are those of the author and do not necessarily reflect the views of the National Science Foundation.}
\and
Shangyou Zhang\thanks{Department of Mathematical Sciences, University of Delaware, Newark, DE 19716 (szhang@udel.edu).}}

\maketitle

\begin{abstract}This paper introduces an extension of the Morley element for approximating solutions to biharmonic equations. Traditionally limited to piecewise quadratic polynomials on triangular elements, the extension leverages weak Galerkin finite element methods to accommodate higher degrees of polynomials and the flexibility of general polytopal elements. By utilizing the Schur complement of the weak Galerkin method, the extension allows for fewest local degrees of freedom while maintaining sufficient accuracy and stability for the numerical solutions. The numerical scheme incorporates locally constructed weak tangential derivatives and weak second order partial derivatives, resulting in an accurate approximation of the biharmonic equation. Optimal order error estimates in both a discrete $H^2$ norm and the usual $L^2$ norm are established to assess the accuracy of the numerical approximation. Additionally, numerical results are presented to validate the developed theory and demonstrate the effectiveness of the proposed extension.
\end{abstract}



\section{Introduction}
This paper is concerned with the new development of high order Morley elements for the biharmonic equation by using the weak Galerkin (WG) method. For simplicity, we consider the biharmonic equation that seeks an unknown function $u$ satisfying
\begin{equation}\label{model-problem}
\begin{split}
\Delta^2u&=g, \quad\mbox{in}~\O, \\
        u&=\zeta,\quad \mbox{on}~\pa\O,\\
\frac{\pa u}{\pa\textbf{n}}&=\xi,\quad\mbox{on}~\pa\O,
\end{split}
\end{equation}
where $\Omega\subset\mathbb R^d(d=2,3)$ is a bounded polytopal domain with Lipschitz continuous boundary $\pa\O$, and $\textbf{n}$ is the unit outward normal vector to $\pa\O.$

A weak formulation of \eqref{model-problem} seeks $u\in H^2(\O)$ satisfying $u|_{\pa\O}=\zeta$ and $\frac{\pa u}{\pa\textbf{n}}|_{\pa\O}=\xi$ such that
\begin{equation}\label{weakform}
\sum_{i,j=1}^d(\pa_{ij}^2u,\pa_{ij}^2v)=(g,v),~~~~\forall v\in H_0^2(\O),
\end{equation}
where $H_0^2(\O)=\{v\in H^2(\O):v|_{\pa\O}=0, \nabla v |_{\pa\O}=\textbf{0}\}$.

The $H^2$-conforming finite element method for the biharmonic equation is well-known but requires a $C^1$-continuity of piecewise polynomials on simplicial elements, which poses practical difficulties. To address this issue, various nonconforming finite element methods were introduced. Among these methods, the Morley element has the fewest degrees of freedom on each triangular element, making it not only a popular research topic but also a practically useful method. Previous works such as \cite{rusa1988, WX2006, WX2012} extended the Morley element to higher dimensions. Other works, including \cite{PS2013, WX2007-2, WXH2006, HI2023, MC2006}, proposed generalizations of the Morley element for different types of meshes. Parallel algorithms and multigrid methods for the Morley element were developed in \cite{CHL2001, HH2011, RS2002, SX1998}. Since then, rapid progress has been made in various numerical methods for the biharmonic equation on polytopal meshes, such as discontinuous Galerkin finite element methods \cite{CDG2009, MSB2007, YZ2022}, virtual element methods \cite{AMV2018, CH2020}, and weak Galerkin  methods \cite{YZ2020, WW_bihar-2014, WW_HWG-2015, WYWM2013, BGZ2020, DR2022, WWL-2022, WWL-2022-bihar, ZXW2023}.  The WG finite element method  was first proposed  for second-order elliptic problems in \cite{wy}. The WG method is a natural generalization of classical finite element methods as it relaxes the continuity requirement for the approximating functions. This weak continuity of the numerical approximation allows for high flexibility in constructing weak finite elements with any desired order of convergence. To the best of our knowledge, no high-order extension has been developed that combines the advantages of the Morley element, including its minimal degrees of freedom, with the ability to handle general polytopal partitions.

The objective of this paper is to present a high-order generalization of the Morley element using the weak Galerkin method. Inspired by the de Rham complexes for weak Galerkin spaces \cite{WWYZ-JCAM2021}, we propose innovations to the original weak finite element procedures. These innovations involve the introduction of additional approximating functions defined on the $(d-2)$-dimensional sub-polytopes and $(d-1)$-dimensional sub-polytopes of $d$-dimensional polytopal elements, resulting in a reduction of the degrees of freedom. To enhance the numerical scheme, we incorporate a locally designed weak tangential derivative operator and a weak second-order partial derivative operator. Furthermore, we establish optimal order error estimates for the resulting numerical approximations in both the energy norm and the $L^2$ norm.

The main contributions of this paper can be summarized as follows. Firstly, unlike the original Morley element, the proposed WG extension allows for higher-order polynomial approximation with the local minimum number of degrees of freedom, while also being applicable to general polytopal elements. This extension broadens the scope of problems that can be effectively addressed. Secondly, in comparison to existing results on WG methods, we introduce a novel technique within the WG framework that significantly reduces the number of unknowns. This advancement enhances the efficiency and computational feasibility of the method. Finally, the versatility of the new WG method enables its application to various modeling problems, including those that involve the Hessian operator in their weak formulation.

The paper is structured as follows. In Section \ref{Section:WeakHessian}, we provide a review of the definitions of the discrete weak tangential derivative and the discrete weak second-order partial derivatives. Section \ref{Section:numerical scheme} presents the weak Galerkin scheme and introduces its Schur complement. Section \ref{Section:seu} establishes the solution existence and uniqueness for this new scheme. Section \ref{Section:error-equation} is devoted to the derivation of an error equation for the weak Galerkin scheme, providing insights into the accuracy of the method.
In Section \ref{technique-estimate}, we present some technical results that are utilized in the subsequent section. Section \ref{Section:EE} is dedicated to establishing error estimates for the numerical approximation, considering both the energy norm and the $L^2$ norm. Finally, in Section \ref{Section:NE}, we present numerical results that demonstrate the effectiveness of the developed theory.

This paper will follow the standard notations for the Sobolev space. For an open bounded domain $D \subset\mathbb R^d$ with Lipschitz continuous boundary $\pa D$, we denote by $\|\cdot\|_{s,D}$ and $|\cdot|_{s,D}$ the norm and semi-norm in the Sobolev space $H^s(D)$ for any $s\geq0$. When $s=0$, we use $(\cdot,\cdot)$ and $|\cdot|_{D}$ to denote the usual integral inner product and semi-norm, respectively. The subscript will be omitted when $D=\O$. Moreover, we use ``$A\lesssim B$'' to denote the inequality ``$A\leq CB$'' where $C$ stands for a generic constant independent of the meshsize or the functions appearing in the inequality.

\section{Discrete weak derivatives}\label{Section:WeakHessian}

Let ${\cal T}_h$ be a polytopal partition of $\Omega$ satisfying the shape regular assumptions described in \cite{WY-ellip_MC2014}. For $T\in{\cal T}_h$, denote by $\pa T$ the boundary of $T$ consisting of $(d-1)$-dimensional polytopal elements (called ``face" for simplicity). For each face $\F\subset\pa T$, denote by $\pa\F$ the boundary of $\F$ consisting of $(d-2)$-dimensional polytopal elements (called ``edge" for simplicity). Denote by ${\F}_h$ the set of all faces for all elements in ${\cal T}_h$ and ${\F}_h^0={\F}_h\setminus\pa\O$ the set of all interior faces. Analogously, denote by $\mathcal{E}_h$ the set of all edges for all elements in ${\cal T}_h$ and $\mathcal{E}_h^0=\mathcal{E}_h\setminus\pa\O$ the set of all interior edges. Moreover, denote by $h_T$ the meshsize of $T$ and $h=\max_{T\in{\cal T}_h}h_T$ the meshsize of ${\cal T}_h$. For any given integer $r\geq0$, denote by $P_r(T)$ and $P_r(\pa T)$ the space of polynomials on $T$ and $\pa T$ with degrees no more than $r$, respectively.

For each element $T\in{\cal T}_h$, we introduce a weak function $v=\{v_0,v_{b,e},v_{b,f},{v_n}\bn_f\}$, where $v_0$ represents the value of $v$ in the interior of $T$, $v_{b,e}$ and $v_{b,f}$ represent the values of $v$ on the edge $e$ and face $\F$ respectively, $\bn_f$ is the unit outward normal vector to $\F$, and ${v_n}$ represents the normal derivative of $v$ on $\pa T$ along the direction $\bn_f$.

 For any given integer $k\geq 3$, denote by $V_{k}(T)$ the discrete space of local weak functions given by
\begin{equation*}
\begin{split}
V_{k}(T)
=\{\{v_0,v_{b,e},v_{b,f},v_n\bn_f\}: ~&v_0\in P_k(T), v_{b,e}\in P_{k-2}(e),v_{b,f}\in P_{k-3}(\F), \\ & ~ v_n\in P_{k-2}(\F),
 \F\subset\pa T, e\subset\pa\F\}.
\end{split}
\end{equation*} It should be pointed out that $v_{b,e}=const$ from problems in 2D.

On each face $\F$, we introduce a finite element space ${{\W}}_{k-2}(\F)$ as polynomial vectors of degree $k-2$ tangential to $\F$:
$$
{{\W}}_{k-2}(\F)=\{\pmb{\psi}:~ \pmb{\psi}\in [P_{k-2}(\F)]^d,\ \pmb{\psi}\cdot\bn_f=0\}.
$$

\begin{definition}\label{def2.1}\cite{WWYZ-JCAM2021}(Discrete weak tangential derivative)
The discrete weak tangential derivative for any weak function $v\in V_{k}(T)$, denoted by $\nabla_{w,\pmb{\tau},k-2,T}v$, is defined as the unique polynomial in ${\W}_{k-2}(\F)$ satisfying
\begin{eqnarray}\label{discrete weak gradient-1}
\langle\nabla_{w,\pmb{\tau},k-2,T}v,\pmb{\psi}\times\pmb{n}_{f}\rangle_{\F}
=-\langle v_{b,f},(\nabla\times\pmb{\psi})\cdot\pmb{n}_{f}\rangle_{\F}+\langle v_{b,e},\pmb{\psi}\cdot{\pmb{\tau}\rangle_{\pa\F}}
\end{eqnarray}
 for all $\pmb{\psi}\in {\W}_{k-2}(\F)$. Here, $\pmb{\tau}$ represents the tangential unit vector on $\pa\F$  that is set such that $\pmb{\tau}$ and $\pmb{n}_{f}$ obey the right hand rule.
\end{definition}

With the normal derivative {$v_n$ and} the discrete weak tangential derivative $\nabla_{w,\pmb{\tau},k-2,T}v$, we can define the weak gradient of $v$  on the face $\F$ as follows:
\begin{eqnarray}\label{decomposition}
 \pmb{v_g}&=v_n\bn_f+\nabla_{w,\pmb{\tau},k-2,T}v.
\end{eqnarray}

\begin{definition}\cite{WW_bihar-2014} (Discrete weak second {order partial derivative}) \label{discrete weak partial derivetive-0}
For any $v\in V_{k}(T)$, the discrete weak second order partial derivative, denoted by $\pa_{ij,w,k-2,T}^{2}v$, is defined as a unique polynomial in $P_{k-2}(T)$ satisfying
\begin{eqnarray}\label{discrete weak partial derivetive}
(\pa_{ij,w,k-2,T}^{2}v,\varphi)_{T}=(v_{0},\pa_{ji}^{2}\varphi)_{T}-\langle v_{b,f}n_{i},\pa_{j}\varphi \rangle_{\pa T}+\langle v_{gi},\varphi {n_{j}\rangle_{\pa T}}
\end{eqnarray}
for any $\varphi\in P_{k-2}(T)$. Here, $\bn_f=(n_1,\ldots,n_d)$ represents the unit outward normal vector to $\pa T$, and $v_{gi}$ is the $i$-th component of the vector $\pmb{v_{g}}$ given in \eqref{decomposition}.
\end{definition}

By utilizing the integration by parts to the first term on the right-hand side of \eqref{discrete weak partial derivetive} we obtain
\begin{equation}\label{error equation-13}
\begin{split}
(\pa_{ij,w,k-2,T}^2v,\varphi)_T
=&(\pa_{ij}^2v_0,\varphi)_T+\langle(v_0-v_{b,f})n_i,\pa_j\varphi\rangle_{\pa T}-\langle\pa_iv_0-v_{gi},\varphi n_j\rangle_{\pa T}
\end{split}
\end{equation}
for any $\varphi\in P_{k-2}(T)$.

\section{Weak Galerkin schemes}\label{Section:numerical scheme}

We construct a global finite element space $V_h$ by patching $V_{k}(T)$ over all the elements $T\in{\cal T}_h$ through common values $v_{b,e}$ on $\mathcal{E}_h^0$, $v_{b,f}$ and $v_n\bn_f$ on ${\F}_h^0$; i.e.,
$$
V_h=\{v=\{v_0,v_{b,e},v_{b,f},v_n\bn_f\}:~v|_T\in V_k(T),~T\in{\cal T}_h\},
$$
Denote by $V_h^0$ the subspace of $V_h$ given by
$$
V_h^0=\{v:v\in V_h,~v_{b,e}|_e=0,~v_{b,f}|_{\F}=0,~v_n|_{\F}=0,~e\subset\pa \O, ~\F\subset\pa\O\}.
$$

For convenience, denote by $\nabla_{w,\pmb{\tau}}v$ the discrete weak tangential derivative $\nabla_{w,\pmb{\tau},k-2,T}v$   and $\pa^2_{ij,w}v$ the discrete weak second order partial derivative $\pa_{ij,w,k-2,T}^{2}v$; i.e.,
$$
(\nabla_{w,\pmb{\tau}}v)|_T=\nabla_{w,\pmb{\tau},k-2,T}(v|_T),\quad(\pa^2_{ij,w}v)|_T=\pa_{ij,w,k-2,T}^2(v|_T),\quad v\in V_h.
$$

Denote by $Q_b$, $Q_f$ and $Q_n$ the usual $L^2$ projection operators onto $P_{k-2}(e)$, $P_{k-3}(\F)$ and $P_{k-2}(\F)$, respectively.  In $V_h\times V_h$, we introduce the following bilinear forms:
\begin{equation*}\label{stabilizer}
\begin{split}
(\pa^2_{w}w,\pa^2_{w}v)&=\sum_{T\in{\cal T}_h}\sum_{i,j=1}^{d}(\pa_{ij,w}^{2}w,\pa_{ij,w}^{2}v)_T,\\
s(w,v)=&\sum_{T\in {\cal T}_h}h_T^{-2}\langle Q_b w_0-w_{b,e}, Q_b v_0-v_{b,e}\rangle_{\pa \F}\\
&+\sum_{T\in {\cal T}_h}h_T^{-3}\langle Q_f w_0-w_{b,f}, Q_{f} v_0-v_{b,f}\rangle_{\pa T}\\
&+\sum_{T\in {\cal T}_h}h_T^{-1}\langle Q_n(\nabla w_0)\cdot\pmb{n}_{f}-w_n, Q_n(\nabla v_0)\cdot\pmb{n}_{f}-v_n\rangle_{\pa T}\\
&+\delta_{k,3}\sum_{T\in {\cal T}_h}h_T^{-1}\langle Q_nD_{\pmb{\tau}}w_0-\nabla_{w,\pmb{\tau}}w,
Q_nD_{\pmb{\tau}}v_0-\nabla_{w,\pmb{\tau}}v\rangle_{\pa T},\\
a_s(w,v)&=(\pa^2_{w}w,\pa^2_{w}v)+s(w,v),
\end{split}
\end{equation*}
   where $Q_nD_{\pmb{\tau}}w_0=Q_n(\bn_{f}\times(\nabla w_0\times\bn_{f}))$ and $\delta_{k,3}$ is the usual Kronecker's delta with value $1$ when $k=3$ and value $0$ otherwise.

\begin{WG}
A numerical approximation for the model equation \eqref{model-problem} based on the weak formulation \eqref{weakform} can be obtained by seeking $u_{h}=\{u_{0},u_{b,e},u_{b,f},u_n\bn_f\}\in V_{h}$ satisfying $u_{b,e}=Q_b\zeta$ on $e\subset\pa\O$, $u_{b,f}=Q_f\zeta$ and $u_n= Q_n\xi$ on $\F\subset\pa\O$ and the following equation
\begin{equation}\label{WG-scheme}
a_s(u_{h},v)=(g,v_0),\qquad\forall v\in V_h^0.
\end{equation}
\end{WG}

\medskip

One may apply the Schur complement approach to the weak Galerkin scheme \eqref{WG-scheme}, yielding an equivalent formulation with reduced number of unknowns in the resulting linear system. More specifically, the Schur complement for \eqref{model-problem} seeks $u_h=\{D(u_{b,e},u_{b,f},u_n,g),u_{b,e},u_{b,f},u_n\bn_f\}\in V_h$ such that $u_{b,e}=Q_b\zeta$ on $e\subset\pa \O$, $u_{b,f}=Q_f\zeta$ and $u_n= Q_n\xi$ on $\F\subset\pa\O$ satisfying
\begin{equation}\label{WG-schemee}
a_s(\{D(u_{b,e},u_{b,f},u_n,g),u_{b,e},u_{b,f},u_n\bn_f\},v)=0
\end{equation}
for all $v=\{0,v_{b,e},v_{b,f},{v_n\bn_f\}\in V_h^0}$,
where $u_0=D(u_{b,e},u_{b,f},u_n,g)$ is obtained by solving the following equation
\begin{equation}\label{WG-schemeee}
a_s(\{u_0,u_{b,e},u_{b,f},u_n\bn_f\},v)=(g,v_0)
\end{equation}
for all $v=\{v_0,0,{0,\textbf{0}\}\in V_h^0.}$

\begin{remark}
The weak Galerkin scheme \eqref{WG-scheme} is equivalent to its Schur complement \eqref{WG-schemee}-\eqref{WG-schemeee}. The  proof is similar to that in \cite{Eff-MWY2017}. As an illustration, when $k=3$, the degrees of freedom on a pentagonal element and a hexahedral element are shown in Figure \ref{polygonal-element}.
 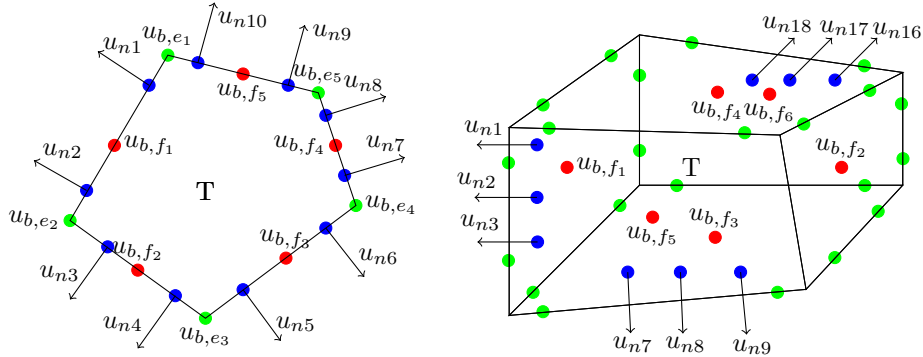
\begin{figure}[h]
\begin{center}
\begin{tikzpicture}
\coordinate (B1) at (-7.5,2.5); \filldraw[green] (B1)circle(0.08);
\coordinate (B2) at (-8.8, 0.3);  \filldraw[green] (B2)circle(0.08);
\coordinate (B3) at (-7, -1); \filldraw[green] (B3) circle(0.08);
\coordinate (B4) at (-5,0.5);\filldraw[green] (B4)circle(0.08);
\coordinate (B5) at (-5.5,2.0);\filldraw[green] (B5)circle(0.08);
\draw node[above] at (B1) {$u_{b,e_1}$}; \draw node[left] at (B2){$u_{b,e_2}$};
\draw node[below] at (B3){$u_{b,e_3}$}; \draw node[right] at (B4) {$u_{b,e_4}$}; \draw node[above] at (B5) {$u_{b,e_5}$};
\coordinate (M1) at (-8.21,1.3); \coordinate (M2) at (-7.9,-0.36); \coordinate (M3) at (-5.93,-.2);
\coordinate (M4) at (-5.27,1.3);  \coordinate (M5) at (-6.5,2.25);  \coordinate (MM) at (0.1,0.15);
\filldraw[red] (M1) circle(0.08);\filldraw[red] (M2) circle(0.08);\filldraw[red] (M3) circle(0.08);
\filldraw[red] (M4) circle(0.08);\filldraw[red] (M5) circle(0.08);
\draw node[right] at (M1) {$u_{b,f_1}$}; \draw node[above] at (M2){$u_{b,f_2}$};
\draw node[above] at (M3){$u_{b,f_3}$}; \draw node[left] at (M4) {$u_{b,f_4}$}; \draw node[below] at (M5) {$u_{b,f_5}$};\coordinate (BM9) at (-5.9,2.1); \filldraw[blue] (BM9) circle(0.08);\coordinate (M9end) at (-5.68,2.9);
\coordinate (ne9) at (-5.7,2.8); \draw node[right] at (ne9) {$u_{n9}$}; \draw[->] (BM9)--(M9end);
\coordinate (BM10) at (-7.1,2.4); \filldraw[blue] (BM10) circle(0.08);\coordinate (M10end) at (-6.88,3.2);
\coordinate (ne10) at (-6.5,3.2); \draw node[below] at (ne10) {$u_{n10}$}; \draw[->] (BM10)--(M10end);
\coordinate (BM1) at (-7.75,2.1); \filldraw[blue] (BM1) circle(0.08);\coordinate (M1end) at (-8.43,2.55);
\coordinate (ne1) at (-8.1,2.4); \draw node[above] at (ne1) {$u_{n1}$}; \draw[->] (BM1)--(M1end);
\coordinate (BM2) at (-8.58,0.7); \filldraw[blue] (BM2) circle(0.08);\coordinate (M2end) at (-9.28,1.1);
\coordinate (ne2) at (-8.9,1.0); \draw node[above] at (ne2) {$u_{n2}$}; \draw[->] (BM2)--(M2end);
\coordinate (BM3) at (-8.3,-0.05); \filldraw[blue] (BM3) circle(0.08);\coordinate (M3end) at (-8.8,-0.75);
\coordinate (ne3) at (-8.55,-0.4); \draw node[left] at (ne3) {$u_{n3}$}; \draw[->] (BM3)--(M3end);\coordinate (BM4) at (-7.4,-0.7); \filldraw[blue] (BM4) circle(0.08);\coordinate (M4end) at (-7.9,-1.4);
\coordinate (ne4) at (-8.1,-1.3); \draw node[above] at (ne4) {$u_{n4}$}; \draw[->] (BM4)--(M4end);
\coordinate (BM5) at (-6.5,-0.62); \filldraw[blue] (BM5) circle(0.08);\coordinate (M5end) at (-6.0,-1.3);
\coordinate (ne5) at (-5.8,-1.3); \draw node[above] at (ne5) {$u_{n5}$}; \draw[->] (BM5)--(M5end);
\coordinate (BM6) at (-5.4,0.2); \filldraw[blue] (BM6) circle(0.08);\coordinate (M6end) at (-4.89,-0.45);
\coordinate (ne6) at (-4.7,-0.4); \draw node[above] at (ne6) {$u_{n6}$}; \draw[->] (BM6)--(M6end);
\coordinate (BM8) at (-5.4,1.7); \filldraw[blue] (BM8) circle(0.08);\coordinate (M8end) at (-4.6,1.95);
\coordinate (ne8) at (-4.9,1.9); \draw node[above] at (ne8) {$u_{n8}$}; \draw[->] (BM8)--(M8end);
\coordinate (BM7) at (-5.15,0.9); \filldraw[blue] (BM7) circle(0.08);\coordinate (M7end) at (-4.35,1.14);
\coordinate (ne7) at (-4.6,1.1); \draw node[above] at (ne7) {$u_{n7}$}; \draw[->] (BM7)--(M7end);
\coordinate (center) at (-7,0.7);\draw node at (center) {T};
\draw node at (center) {T}; \draw (B1)--(B2)--(B3)--(B4)--(B5)--cycle;\coordinate (A1) at (-2,0,-2); \coordinate (A2) at (1.5, 0,-2);
\coordinate (A3) at (1.5, 1.5,-2);\coordinate (A4) at (-2,2,-2);
\coordinate (A5) at (-2,0,2.5); \coordinate (A6) at (1.6, 0,1.6);
\coordinate (A7) at (1.6, 2.4,2.5); \coordinate (A8) at (-2,2.5,2.5);
\coordinate (SB1) at (-2.3,1.8,-1.8); \filldraw[green] (SB1) circle(0.08);
\coordinate (SB2) at (-2.05,2.3,1.2); \filldraw[green] (SB2) circle(0.08);
\coordinate (SB3) at (0.7,1.3,-2.75); \filldraw[green] (SB3) circle(0.08);
\coordinate (SB4) at (-1.4,1.8,-2.25); \filldraw[green] (SB4) circle(0.08);
\coordinate (SB5) at (1.6,1.8,-0.6); \filldraw[green] (SB5) circle(0.08);
\coordinate (SB6) at (1.57,2.2,1.6); \filldraw[green] (SB6) circle(0.08);
\coordinate (SB7) at (0.9,2.2,1.9); \filldraw[green] (SB7) circle(0.08);
\coordinate (SB8) at (-1.7,2.26,1.9); \filldraw[green] (SB8) circle(0.08);
\coordinate (SB9) at (-2.24,1.8,1.9); \filldraw[green] (SB9) circle(0.08);
\coordinate (SB10) at (-2.24,0.5,1.9); \filldraw[green] (SB10) circle(0.08);
\coordinate (SB11) at (-1.96,1.5,-1.9); \filldraw[green] (SB11) circle(0.08);
\coordinate (SB13) at (-1.49,0.5,-0.0); \filldraw[green] (SB13) circle(0.08);
\coordinate (SB46) at (-1.0,0.5,-0.69); \filldraw[green] (SB46) circle(0.08);
\coordinate (SB47) at (1.5,0.5,-0.69); \filldraw[green] (SB47) circle(0.08);
\coordinate (SB14) at (-1.49,0.5,3.0); \filldraw[green] (SB14) circle(0.08);
\coordinate (SB17) at (2.06,0.5,1.8); \filldraw[green] (SB17) circle(0.08);
\coordinate (SB18) at (2.02,0.5,0.2); \filldraw[green] (SB18) circle(0.08);
\coordinate (SB12) at (-1.96,0.5,-1.9); \filldraw[green] (SB12) circle(0.08);
\coordinate (SB15) at (-1.1,0.5,3.67); \filldraw[green] (SB15) circle(0.08);
\coordinate (SB16) at (1.76,0.5,3); \filldraw[green] (SB16) circle(0.08);
\coordinate (SB19) at (2.35,1.2,0.2); \filldraw[green] (SB19) circle(0.08);
\coordinate (SB20) at (3.1,2.7,2.2); \filldraw[green] (SB20) circle(0.08);\coordinate (M1485mid) at (-2.0,1.2,0.5);  \filldraw[red] (M1485mid)circle(0.08);  \draw node[right] at (M1485mid) {$u_{b,f_1}$};
\coordinate (M1562mid) at (-0.3,0.0,-0.2);  \filldraw[red] (M1562mid)circle(0.08);\draw node[above] at (M1562mid) {$u_{b,f_3}$};
\coordinate (M2376mid) at (1.65,1.2,0.5);    \filldraw[red] (M2376mid)circle(0.08);\draw node[above] at (M2376mid) {$u_{b,f_2}$};
\coordinate (M3487mid) at (0,2.2,0.5);  \filldraw[red] (M3487mid)circle(0.08); \draw node[below] at (M3487mid) {$u_{b,f_4}$};
\coordinate (M5678mid) at (0.1,1.5,3.0);   \filldraw[red] (M5678mid)circle(0.08);\draw node[below] at (M5678mid) {$u_{b,f_5}$};
 \coordinate (M1234mid) at (-0.5,0.98,-2.6);   \filldraw[red] (M1234mid)circle(0.08); \draw node[below] at (M1234mid) {$u_{b,f_6}$};
\coordinate (SM1485mid1) at (-2.4,1.5,0.5);  \filldraw[blue] (SM1485mid1)circle(0.08);
\coordinate (SM1end) at (-3.2,1.5,0.5); \coordinate (ne1) at (-3.1,1.5,0.5); \draw node[above] at (ne1) {$u_{n1}$};
 \draw[->] (SM1485mid1)--(SM1end);
\coordinate (SM1485mid2) at (-2.4,0.8,0.5);  \filldraw[blue] (SM1485mid2)circle(0.08);
\coordinate (SM2end) at (-3.23,0.8,0.5); \coordinate (ne2) at (-3.2,0.8,0.5); \draw node[above] at (ne2) {$u_{n2}$};
 \draw[->] (SM1485mid2)--(SM2end);\coordinate (SM1485mid3) at (-2.2,0.4,1);  \filldraw[blue] (SM1485mid3)circle(0.08);
\coordinate (SM3end) at (-3.0,0.4,1); \coordinate (ne3) at (-3.0,0.4,0.9); \draw node[above] at (ne3) {$u_{n3}$};
 \draw[->] (SM1485mid3)--(SM3end);
\coordinate (SM3487mid1) at (1.6,2.4,0.6);  \filldraw[blue] (SM3487mid1)circle(0.08);
\coordinate (SM10end) at (2.2,3,0.7); \coordinate (ne16) at (2.55,2.95,0.9); \draw node[above] at (ne16) {$u_{n16}$};
 \draw[->] (SM3487mid1)--(SM10end);
\coordinate (SM3487mid2) at (1.0,2.4,0.6);  \filldraw[blue] (SM3487mid2)circle(0.08);
\coordinate (SM11end) at (1.6,3,0.7); \coordinate (ne17) at (1.96,3.1,1.2); \draw node[above]at (ne17) {$u_{n17}$};
\draw[->] (SM3487mid2)--(SM11end);
\coordinate (SM3487mid3) at (0.5,2.4,0.6);  \filldraw[blue] (SM3487mid3)circle(0.08);
\coordinate (SM12end) at (1.1,3,0.7); \coordinate (ne18) at (1.2,3.15,1.2); \draw node[above] at (ne18) {$u_{n18}$};
 \draw[->] (SM3487mid3)--(SM12end);
\coordinate (M1562mid1) at (1.5,1.0,3.6);  \filldraw[blue] (M1562mid1)circle(0.08);
\coordinate (SM13end) at (2.85,1.44,6.9); \coordinate (ne9) at (2.43,1,5.6); \draw node[below] at (ne9) {$u_{n9}$};
 \draw[->] (M1562mid1)--(SM13end);
 \coordinate (M1562mid2) at (0,1.0,3.6);  \filldraw[blue] (M1562mid2)circle(0.08);
\coordinate (SM14end) at (0.5,0.65,4.8); \coordinate (ne8) at (0.8,1,5.5); \draw node[below] at (ne8) {$u_{n7}$};
 \draw[->] (M1562mid2)--(SM14end);
  \coordinate (M1562mid3) at (0.7,1.0,3.6);  \filldraw[blue] (M1562mid3)circle(0.08);
\coordinate (SM15end) at (2.15,1.6,7.28); \coordinate (ne15) at (1.5,1,5.5); \draw node[below] at (ne15) {$u_{n8}$};\draw[->] (M1562mid3)--(SM15end);
\draw (A1)--(A4)--(A8)--(A5)--cycle;
\draw (A2)--(A3)--(A7)--(A6)--cycle;
\draw (A1)--(A5)--(A6)--(A2)--cycle;
\draw (A3)--(A4)--(A8)--(A7)--cycle;
\draw (A1)--(A2)--(A3)--(A4)--cycle;
\draw (A5)--(A6)--(A7)--(A8)--cycle;
\coordinate (M12end) at (4,0,-2);
\coordinate (M14end) at (-2,4,-2);
\coordinate (M15end) at (-2,0.0,4.5);
\coordinate (center) at (-0.6,1.2,0.5);
\draw node[right] at (center) {T};
\end{tikzpicture}
\caption{Local degrees of freedom for the finite element space $V_3(T)$ on a pentagonal element (left) and a hexahedral element (right).}\label{polygonal-element}
\end{center}
\end{figure}
\end{remark}


\section{Solution existence and uniqueness}\label{Section:seu}
On each element $T\in{\cal T}_{h}$, denote by $Q_0$ the usual $L^2$ projection operator   onto $P_{k}(T)$. For any $\phi\in H^2(\O)$, let
$$
Q_{h}\phi=\{Q_{0}\phi,Q_b\phi,Q_f\phi,Q_n(\nabla\phi\cdot\bn_f)\bn_f\}.
$$
Similarly, denote by $\mathbb{Q}_h$ the $L^2$ projection operator onto $P_{k-2}(T).$

\begin{lemma}\label{commutative properties}
For any $\phi\in H^{2}(T)$, the following commutative property holds true
\begin{eqnarray}
\label{comu2}
\nabla_{w,\pmb{\tau}}Q_{h}\phi&=&Q_n(\bn_f\times(\nabla\phi\times\bn_f)),\\
\label{comu1}
\pa_{ij,w}^{2} (Q_{h}\phi) &=&\mathbb{Q}_h(\pa_{ij}^{2}\phi),\quad i,j=1,\ldots,d.
\end{eqnarray}
\end{lemma}

\begin{proof} First of all, the identity \eqref{comu2} has been established in \cite{WWYZ-JCAM2021}. Hence, the gradient representation \eqref{decomposition} for ${(Q_h\phi)}_g$ has the following form
\begin{eqnarray}\label{decomposition-new}
(Q_h\phi)_g&=Q_n(\nabla\phi\cdot\bn_f)\bn_f+Q_n(\bn_f\times(\nabla\phi\times\bn_f))\\
&=Q_n(\nabla\phi).
\end{eqnarray}
In other words, the weak gradient of $Q_h\phi$ is the $L^2$ projection of the classical gradient of $\phi$ on each face $\F\subset\pa T$.
Thus, from \eqref{discrete weak partial derivetive} and the usual integration by parts we obtain
\begin{equation*}
\begin{split}
 &(\pa_{ij,w}^{2}(Q_h\phi),\varphi)_{T}\\
=&(Q_0\phi,\pa_{ji}^{2}\varphi)_{T}-\langle Q_f\phi n_i,\pa_j\varphi\rangle_{\pa T}
  +{\langle Q_n(\nabla\phi)_i,\varphi n_{j}\rangle_{\pa T}}\\
=&(\phi,\pa_{ji}^{2}\varphi)_T-\langle\phi n_i,\pa_{j}\varphi\rangle_{\pa T}+\langle(\nabla\phi)_i,\varphi n_j\rangle_{\pa T}\\
=&(\pa_{ij}^2\phi,\varphi)_T\\
=&(\mathbb{Q}_h(\pa_{ij}^{2}\phi),\varphi)_T
\end{split}
\end{equation*}
for all $\varphi\in P_{k-2}(T)$. This verifies the identity \eqref{comu1}.
\end{proof}


Observe that the bilinear form $a_s(v,v)$ induces a semi-norm in the finite element space $V_h$ given by
\begin{equation}\label{tri-semibar}
\3barv\3bar=\big(a_s(v,v)\big)^{1/2}.
\end{equation}

\begin{lemma}\label{tri-norm}
The semi-norm $\3barv\3bar$ defined by \eqref{tri-semibar} is a norm in the subspace $V_h^0$.
\end{lemma}

\begin{proof}
It suffices to show that $\3bar v\3bar =0$ implies $v=0$. To this end, assume $\3barv\3bar=0$ for some $v\in V_h^0$. From \eqref{tri-semibar} we have $\pa^2_{w}v=0$ and $s(v,v)=0$, which implies $\pa^2_{ij,w}v=0$ for $i,j=1,\ldots,d$ on each $T$, $Q_bv_0=v_{b,e}$ on each $\pa \F$, $Q_fv_0=v_{b,f}$ and $Q_n(\nabla v_0)\cdot\bn_f=v_n$ on each $\pa T$. Thus, on each element $T\in \T_h$ we have $Q_h v_0=v$ so that by using \eqref{comu1}
$$
\partial_{ij}^2 v_0 = \mathbb{Q}_h\partial_{ij}^2 v_0 = \partial_{ij,w}^2 (Q_hv_0) =\partial_{ij,w}^2 v =0,\quad i,j=1,\ldots, d.
$$
Hence, $\nabla v_0=const$ on each $T\in\T_h$. Note that on each face $\F\in\partial T$ we have
$$\nabla v_0=(\nabla v_0\cdot\bn_f)\bn_f+\bn_f\times(\nabla v_0\times\bn_f),
$$
which, together with $Q_n(\bn_f\times(\nabla v_0\times\bn_f))    =\nabla_{w,\pmb{\tau}}v$ and {$Q_n(\nabla v_0)\cdot\bn_f =v_n$}, gives rise to $\nabla v_0=v_n \bn_f+\nabla_{w,\pmb{\tau}}v$ on each face $\F\in\F_h$ and hence $\nabla v_0\in C^0(\O)$. Next, with $v_{b,f}=0$ on each $\F\subset\pa\O$ and $v_{b,e}=0$ on each $e\subset\pa\O$ we have from \eqref{discrete weak gradient-1} that $\nabla_{w,\pmb{\tau}}v=0$ on each $\F\subset\pa\O$.
This, together with $v_n=0$ on each $\F\subset\pa\O$, gives $\nabla v_0=0$ on $\F\subset\pa\O$ and further $\nabla v_0=0$ in the domain $\O$ since $\nabla v_0=const$  on each $T$ and $\nabla v_0\in C^0(\O)$. Hence, $v_n=0$ on each $\F$ and $v_0=const$ on each $T$. This further leads to $v_0=Q_bv_0=v_{b,e}$ on each $\pa \F$ and $v_0=Q_fv_0=v_{b,f}$ on each $\pa T$, and hence $v_0\in C^0(\O)$. From $v_{b,e}=0$ on $e\subset\pa\O$ and $v_{b,f}=0$ on each $\F\subset\pa\O$ we have $v_0=0$ in $\O$. Finally, from $v_{b,e}=Q_bv_0$ on each $\pa \F$ and $v_{b,f}=Q_fv_0$ on each $\pa T$ we have $v_{b,e}=0$ on each $\pa \F$ and $v_{b,f}=0$ on each $\pa T$. This completes the proof of the lemma.
\end{proof}

\begin{lemma}\label{unique}
The weak Galerkin scheme \eqref{WG-scheme} has one and only one numerical approximation.
\end{lemma}

\begin{proof}
It suffices to verify the uniqueness of the numerical approximation. To this end, assume that $u_h^{(1)}$ and $u_h^{(2)}$ are two solutions of \eqref{WG-scheme}. It is clear that
\begin{equation}\label{WG-unique}
a_s(u_h^{(1)}-u_h^{(2)},v)=0,~~~\forall v\in V_h^0.
\end{equation}
By letting $v=u_h^{(1)}-u_h^{(2)}\in V_h^0$ in \eqref{WG-unique} we obtain
$$\3baru_h^{(1)}-u_h^{(2)}\3bar=0,$$
which implies $u_h^{(1)}=u_h^{(2)}$ from Lemma \ref{tri-norm}. This completes the proof of the lemma.
\end{proof}

\section{Error equations}\label{Section:error-equation}
Let $u$ be the exact solution of the model equation \eqref{model-problem}  and $u_h\in V_h$ be the numerical solution of the WG scheme \eqref{WG-scheme}, respectively. Denote by
\begin{equation}\label{error function}
e_h=Q_hu-u_h
\end{equation}
the error function between the $L^2$ projection of the exact solution and its WG finite element approximation $u_h$.

\begin{lemma}\label{error equation}
The error function $e_h$ defined in \eqref{error function} satisfies the following error equation
\begin{eqnarray}\label{Error-equation}
a_s(e_h,v)=\zeta_u(v),\qquad\forall v\in V_h^0,
\end{eqnarray}
where $\zeta_u(v)$ is given by
\begin{equation}\label{error equation-remainder}
\begin{split}
\zeta_u(v)
=&s(Q_hu,v)+\sum_{T\in {\cal T}_h}\sum_{i,j=1}^d\langle v_0-v_{b,f},\pa_j(\mathbb{Q}_h(\pa^2_{ij}u)-\pa^2_{ij}u) n_i\rangle_{\pa T}\\
 &+\sum_{T\in{\cal T}_h}\sum_{i,j=1}^d\langle\pa_iv_0-v_{gi},(\pa^2_{ij}u-\mathbb{Q}_h(\pa^2_{ij}u)) n_j\rangle_{\pa T}.
\end{split}
\end{equation}
\end{lemma}
\begin{proof}
Let $v\in V_h^0$. On any face $\F\subset\pa\O$, we have $v_{b,f}=0$ and $v_{b,e}=0$ on $e\subset \pa\F$. Thus, from \eqref{discrete weak gradient-1} we have
$$
\langle\nabla_{w,\pmb{\tau}}v,\pmb{\psi}\times\pmb{n}_{f}\rangle_{\F}=-\langle v_{b,f},(\nabla\times\pmb{\psi})\cdot\pmb{n}_{f}\rangle_{\F}+\langle v_{b,e},\pmb{\psi}\cdot\pmb{\tau}\rangle_{\pa \F}=0,
$$
for any $\pmb{\psi}\in \W_{k-2}(\F)$. Hence, $\nabla_{w,\pmb{\tau}}v=0$ on $\pa\O$. This, together with \eqref{decomposition} and $v_n=0$ on $\pa\O$, gives rise to $\pmb{v}_g=0$
on $\pa\O$.

By testing the model equation \eqref{model-problem} against $v_0$ and then using the usual integration by parts, we have
\begin{equation}\label{error equation-1}
\begin{split}
(g,v_0)
=&\sum_{T\in{\cal T}_h}(\Delta^2u,v_0)_T\\
=&\sum_{T\in{\cal T}_h}\sum_{i,j=1}^d(\pa^2_{ij}u,\pa^2_{ij}v_0)_T-\langle\pa^2_{ij}u,\pa_iv_0  n_j\rangle_{\pa T}
  +\langle\pa_j(\pa^2_{ij}u)  n_i,v_0\rangle_{\pa T}\\
=&\sum_{T\in{\cal T}_h}\sum_{i,j=1}^d(\pa^2_{ij}u,\pa^2_{ij}v_0)_T-\langle\pa^2_{ij}u,(\pa_iv_0-v_{gi})  n_j\rangle_{\pa T}\\
 &+\langle\pa_j(\pa^2_{ij}u)  n_i,v_0-v_{b,f}\rangle_{\pa T},
\end{split}
\end{equation}
where we used the fact that
\begin{eqnarray*}
   && \sum_{T\in{\cal T}_h}\sum_{i,j=1}^d \langle\pa^2_{ij}u,  v_{gi}   n_j\rangle_{\pa T}=0,\\
   &&\sum_{T\in{\cal T}_h}\sum_{i,j=1}^d \langle\pa_j(\pa^2_{ij}u)  n_i, v_{b,f}\rangle_{\pa T}=0,
\end{eqnarray*}
and $v_{b,f}=0, \ \bv_{g}=0$ on $\F\subset\pa\O$.

To handle the first term on last line in \eqref{error equation-1}, {we choose} $\varphi=\mathbb{Q}_h(\pa_{ij}^2u)\in P_{k-2}(T)$ in \eqref{error equation-13} and then use Lemma \ref{commutative properties} to obtain
\begin{equation}\label{error equation-3}
\begin{split}
 (\pa_{ij}^2v_0,\pa_{ij}^2u)_T=&(\pa_{ij}^2v_0,\mathbb{Q}_h(\pa_{ij}^2u))_T\\
=&(\pa_{ij,w}^2v,\mathbb{Q}_h(\pa_{ij}^2u))_T-\langle(v_0-v_{b,f})n_i,\pa_j(\mathbb{Q}_h(\pa^2_{ij}u))\rangle_{\pa T}\\
 &+\langle\pa_iv_0-v_{gi},\mathbb{Q}_h(\pa^2_{ij}u) n_j\rangle_{\pa T}\\
=&(\pa_{ij,w}^2v,\pa_{ij,w}^2Q_hu)_T-\langle(v_0-v_{b,f})n_i,\pa_j(\mathbb{Q}_h(\pa^2_{ij}u))\rangle_{\pa T}\\
 &+\langle\pa_iv_0-v_{gi},\mathbb{Q}_h(\pa^2_{ij}u) n_j\rangle_{\pa T}.
\end{split}
\end{equation}
Substituting \eqref{error equation-3} into \eqref{error equation-1} gives
\begin{equation}\label{error equation-remainder-4}
\begin{split}
(g,v_0)
=&\sum_{T\in{\cal T}_h}\sum_{i,j=1}^d(\pa_{ij,w}^2v,\pa_{ij,w}^2Q_hu)_T
  +\langle v_0-v_{b,f},\pa_j(\pa^2_{ij}u-\mathbb{Q}_h(\pa^2_{ij}u)) n_i\rangle_{\pa T}\\
 &+\langle\pa_iv_0-v_{gi},(\mathbb{Q}_h(\pa^2_{ij}u)-\pa^2_{ij}u) n_j\rangle_{\pa T}.
\end{split}
\end{equation}
Subtracting \eqref{WG-scheme} from \eqref{error equation-remainder-4} gives rise to Lemma \ref{error equation}.
\end{proof}

\section{Technical results}\label{technique-estimate}
Note that for any $T\in{\cal T}_h$ and $\phi\in H^1(T)$, the following trace inequality \cite{WY-ellip_MC2014} holds true:
\begin{equation}\label{trace inequality}
\|\phi\|_{\pa T}^2\lesssim h_T^{-1}\|\phi\|_T^2+h_T\|\nabla\phi\|_T^2.
\end{equation}
If $\phi$ is a polynomial on the element $T\in{\cal T}_h$, we have
from the inverse inequality that
\begin{equation}\label{inverse inequality}
\|\phi\|_{\pa T}^2\lesssim h_T^{-1}\|\phi\|_T^2.
\end{equation}

\begin{lemma}\label{error projection}
Assume that ${\cal T}_h$ is a finite element partition satisfying the regular assumptions described  in \cite{WY-ellip_MC2014}. Then, for any $0\leq s\leq2$, the following error estimates \cite{WY-ellip_MC2014,MWY2014} hold true:
\begin{equation}\label{q0}
   \sum_{T\in{\cal T}_h}h_T^{2s}\|\phi-Q_0\phi\|_{s,T}^2\lesssim h^{2(k+1)}\|\phi\|_{k+1}^2,
\end{equation}
 \begin{equation}\label{qh}
 \sum_{T\in{\cal T}_h}\sum_{i,j=1}^dh_T^{2s}\|\pa^2_{ij}\phi-\mathbb{Q}_h(\pa^2_{ij}\phi)\|_{s,T}^2
\lesssim h^{2(k-1)}\|\phi\|_{k+1}^2.
\end{equation}
\end{lemma}

\begin{lemma}\label{PROJECTION-ESTIMATE-68}
For any $v\in V_h$, there holds
\begin{equation}\label{EQ:June19:001}
    \big(\sum_{T\in{\cal T}_h}\sum_{i=1}^dh_T^{-1}\|Q_n(\pa_iv_0)-v_{gi}\|_{\pa T}^2\big)^{\frac{1}{2}}\lesssim\3barv\3bar.
\end{equation}

\end{lemma}
\begin{proof}
From $\nabla v_0=(\nabla v_0\cdot\bn_f)\bn_f+\bn_f\times(\nabla v_0\times\bn_f)$ and \eqref{decomposition}, we have
\begin{equation}\label{0:33}
\begin{split}
 &\sum_{T\in{\cal T}_h}\sum_{i=1}^dh_T^{-1}\|Q_n(\pa_iv_0)-v_{gi}\|_{\pa T}^2\\
=&\sum_{T\in{\cal T}_h}h_T^{-1}\|Q_n(\nabla v_0)-\pmb{v_{g}}\|_{\pa T}^2\\
=&\sum_{T\in{\cal T}_h}h_T^{-1}\|Q_n(\nabla v_0\cdot\bn_f)\bn_f+Q_n(\bn_f\times(\nabla v_0\times\bn_f))-(v_n\bn_f
 +\nabla_{w,\pmb{\tau}}v)\|_{\pa T}^2\\
\lesssim&\sum_{T\in{\cal T}_h}h_T^{-1}\|Q_n(\nabla v_0\cdot\bn_f)-v_n\|_{\pa T}^2
 +h_T^{-1}\|Q_n(\bn_f\times(\nabla v_0\times\bn_f))-\nabla_{w,\pmb{\tau}}v\|_{\pa T}^2\\
\lesssim&\3barv\3bar^2+\sum_{\F\in{ \F}_h}h_T^{-1}\|Q_n(\bn_f\times(\nabla v_0\times\bn_f))-\nabla_{w,\pmb{\tau}}v\|_\F^2.
\end{split}
\end{equation}
Next, from \eqref{discrete weak gradient-1} and the Stokes Theorem we have
 \begin{equation*}
\begin{split}
&|\langle Q_n(\bn_f\times(\nabla v_0\times\bn_f))-\nabla_{w,\pmb{\tau}}v,\pmb{\psi}\times\bn_f\rangle_\F| \\
 = \ & |\langle Q_fv_0-v_{b,f},(\nabla\times\pmb{\psi})\cdot\bn_f\rangle_{\F}+\langle v_{b,e}-Q_bv_0,\pmb{\psi}\cdot\pmb{\tau}\rangle_{\pa \F}|\\
 \leq \ & \|Q_fv_0-v_{b,f}\|_{\F}\|\nabla\times\pmb{\psi}\|_{\F}
 +\|v_{b,e}-Q_bv_0\|_{\pa \F}\|\pmb{\psi}\|_{\pa \F}\\
 \lesssim\ & \|Q_fv_0-v_{b,f}\|_{\F}h_T^{-1}\|\pmb{\psi}\|_{\F}
 +\|v_{b,e}-Q_bv_0\|_{\pa\F}h_T^{-\frac{1}{2}}\|\pmb{\psi}\|_{\F}
\end{split}
\end{equation*}
for all $\pmb{\psi}\in \W_{k-2}(\F)$. Hence,
$$
\|Q_n(\bn_f\times(\nabla v_0\times\bn_f))-\nabla_{w,\pmb{\tau}}v\|_\F\lesssim h_T^{-1}\|Q_fv_0-v_{b,f}\|_{\F}+h_T^{-\frac{1}{2}}\|v_{b,e}-Q_bv_0\|_{\pa \F}.
$$
Substituting the above estimate into \eqref{0:33} gives rise to the desired inequality \eqref{EQ:June19:001}.
\end{proof}

\begin{lemma}\label{PROJECTION-ESTIMATE-6}
For any $v\in V_h$, there yields
\begin{equation}\label{EQ:June19:002}
  \sum_{T\in{\cal T}_h}|v_0|_{2,T}^2\lesssim\3barv\3bar^2.
\end{equation}

\end{lemma}
\begin{proof}
By taking $\varphi=\pa_{ij}^2v_0\in P_{k-2}(T)$ in \eqref{error equation-13} we have
\begin{equation*}
\begin{split}
 &(\pa_{ij,w}^2v,\pa_{ij}^2v_0)_T\\
=&(\pa_{ij}^2v_0,\pa_{ij}^2v_0)_T+\langle(v_0-v_{b,f})n_i,\pa_j(\pa_{ij}^2v_0)\rangle_{\pa T}
  -\langle\pa_iv_0-v_{gi},\pa_{ij}^2v_0 n_j\rangle_{\pa T}\\
=&(\pa_{ij}^2v_0,\pa_{ij}^2v_0)_T+\langle(Q_fv_0-v_{b,f})n_i,\pa_j(\pa_{ij}^2v_0)\rangle_{\pa T}
  -\langle Q_n(\pa_iv_0)-v_{gi},\pa_{ij}^2v_0 n_j\rangle_{\pa T}.
\end{split}
\end{equation*}
Hence,
\begin{equation*}
\begin{split}
\sum_{T\in{\cal T}_h}|v_0|_{2,T}^2
\lesssim&\Big(\sum_{T\in{\cal T}_h}\sum_{i,j=1}^d\|\pa_{ij,w}^2v\|_{T}^2\Big)^{\frac{1}{2}}
 \Big(\sum_{T\in{\cal T}_h}\sum_{i,j=1}^d\|\pa_{ij}^2v_0\|_{T}^2\Big)^{\frac{1}{2}}\\
&+\Big(\sum_{T\in{\cal T}_h}h_T^{-3}\|Q_fv_0-v_{b,f}\|_{\pa T}^2\Big)^{\frac{1}{2}}
 \Big(\sum_{T\in{\cal T}_h}\sum_{i,j=1}^dh_T^3\|\pa_j(\pa_{ij}^2v_0)\|_{\pa T}^2\Big)^{\frac{1}{2}}\\
&+\Big(\sum_{T\in{\cal T}_h}\sum_{i=1}^dh_T^{-1}\|Q_n(\pa_iv_0)-v_{gi}\|_{\pa T}^2\Big)^{\frac{1}{2}}
 \Big(\sum_{T\in{\cal T}_h}\sum_{i,j=1}^dh_T\|\pa_{ij}^2v_0\|_{\pa T}^2\Big)^{\frac{1}{2}}\\
\lesssim&\3barv\3bar\Big(\sum_{T\in{\cal T}_h}|v_0|_{2,T}^2\Big)^{\frac{1}{2}}.
\end{split}
\end{equation*}
This completes the proof of the lemma.
\end{proof}

\begin{lemma}\label{PROJECTION-ESTIMATE-7}
Let $k\ge 3$. For any $v\in V_h$ and $\varphi\in H^{k+1}(\O)$, there holds
\begin{eqnarray}
    |\sum_{T\in{\cal T}_h}\sum_{i,j=1}^d\langle v_0-Q_fv_0,\pa_j(\mathbb{Q}_h(\pa^2_{ij}\varphi)-\pa^2_{ij}\varphi) n_i\rangle_{\pa T}| &\lesssim& h^{k-1}\|\varphi\|_{k+1}\3barv\3bar,\label{esti1}\\
    \label{esti2}
\Big(\sum_{T\in{\cal T}_h}h_T^{-1}\|Q_n(D_{\pmb{\tau}}Q_0\varphi)-\nabla_{w,\pmb{\tau}}Q_h\varphi\|_{\pa T}^2\Big)^{\frac{1}{2}}&\lesssim& h^{k-1}\|\varphi\|_{k+1}.
\end{eqnarray}
\end{lemma}
\begin{proof}
We first note the following identity
\begin{equation*}
\begin{split}
J:=&\sum_{T\in{\cal T}_h}\sum_{i,j=1}^d\langle v_0-Q_fv_0,{\pa_j(\pa^2_{ij}\varphi-\mathbb{Q}_h(\pa^2_{ij}\varphi))n_i}\rangle_{\pa T}\\
=&\sum_{T\in{\cal T}_h}\sum_{i,j=1}^d\langle v_0-Q_fv_0,\pa_j\pa^2_{ij}\varphi n_i\rangle_{\pa T}\\
=&\sum_{T\in{\cal T}_h}\sum_{i,j=1}^d\langle v_0-Q_fv_0,(I-Q_f)\pa_j\pa^2_{ij}\varphi n_i\rangle_{\pa T}\\
=&\sum_{\F\in{\F}_h}\sum_{i,j=1}^d\langle[v_0]-Q_f[v_0],(I-Q_f)\pa_j\pa^2_{ij}\varphi n_{i}\rangle_{\F}.
\end{split}
\end{equation*}
For $k>3$, the finite element space on face $\F$ consists of {linear functions} so that
$$
|\langle[v_0]-Q_f[v_0],(I-Q_f)\pa_j\pa^2_{ij}\varphi n_{i}\rangle_{\F}|\leq Ch^2\|[v_0]\|_{2,\F} \|(I-Q_f)\pa_j\pa^2_{ij}\varphi\|_{0,\F},
$$
which can be used to derive the desired inequality \eqref{esti1} without any difficulty. Here $[v_0]=v_0|_{T_L\cap\F}-v_0|_{T_R\cap\F}$ is the jump of $v_0$ on the face $\F$ shared by two adjacent elements $T_L$ and $T_R$.

For the case of $k=3$, the finite element space on face $\F$ consists of constants only so that
$$
|\langle[v_0]-Q_f[v_0],(I-Q_f)\pa_j\pa^2_{ij}\varphi n_{i}\rangle_{\F}|\leq Ch\|[D_{\pmb{\tau}}v_0]\|_{0,\F} \|(I-Q_f)\pa_j\pa^2_{ij}\varphi\|_{0,\F},
$$
where {$D_{\pmb{\tau}}v_0$ stands} for the tangential derivative on $\F$. It follows from the trace inequalities \eqref{trace inequality}-\eqref{inverse inequality} and the inverse inequality that
\begin{equation*}
\begin{split}
|J|\lesssim&\Big(\sum_{\F\in{\F}_h}h_T^2\|[D_{\pmb{\tau}}v_0]\|_{\F}^2\Big)^{\frac{1}{2}}\\
&\cdot\Big(\sum_{T\in{\cal T}_h}\sum_{i,j=1}^dh_T^{-1}\|(I-Q_f)\pa_j\pa^2_{ij}\varphi\|_{T}^2+h_T|(I-Q_f)\pa_j\pa^2_{ij}\varphi|_{1,T}^2\Big)^{\frac{1}{2}}\\
\lesssim&\Big(\sum_{\F\in{\F}_h}h_T^2\|[D_{\pmb{\tau}}v_0]-Q_n([D_{\pmb{\tau}}v_0])\|_{\F}^2+h_T^2\|Q_n([D_{\pmb{\tau}}v_0])\|_{\F}^2\Big)^{\frac{1}{2}}\\
&\cdot\Big(\sum_{T\in{\cal T}_h}h_T^{2k-5}\|\varphi\|_{k+1,T}^2\Big)^{\frac{1}{2}}\\
\lesssim&\Big(\sum_{\F\in{\F}_h}h_T^4\|[D_{\pmb{\tau\tau}}v_0]\|_{\F}^2+h_T^2\|Q_n([D_{\pmb{\tau}}v_0])-[
\nabla_{w,\pmb{\tau}}v]\|_{\F}^2\Big)^{\frac{1}{2}}h^{k-{\frac{5}{2}}}\|\varphi\|_{k+1}\\
\lesssim&\Big(\sum_{T\in{\cal T}_h}h_T^3|v_0|_{2,T}^2+\sum_{T\in{\cal T}_h}h_T^2\|Q_nD_{\pmb{\tau}}v_0-\nabla_{w,\pmb{\tau}}v\|_{\pa T}^2\Big)^{\frac{1}{2}}
h^{k-{\frac{5}{2}}}\|\varphi\|_{k+1}\\
\lesssim&\Big(\sum_{T\in{\cal T}_h}h_T^3|v_0|_{2,T}^2+h^3\3barv\3bar^2\Big)^{\frac{1}{2}}h^{k-{\frac{5}{2}}}\|\varphi\|_{k+1}\\
\lesssim&\ h^{k-1}\|\varphi\|_{k+1}\3barv\3bar,
\end{split}
\end{equation*}
which completes the proof of \eqref{esti1}.

To verify \eqref{esti2}, we recall that $Q_nD_{\pmb{\tau}}w_0=Q_n(\bn_{f}\times(\nabla w_0\times\bn_{f}))$. Hence, from \eqref{comu2}, the trace inequality \eqref{trace inequality}, and \eqref{q0} we arrive at
\begin{equation*}
\begin{split}
&\sum_{T\in{\cal T}_h}h_T^{-1}\|Q_n(D_{\pmb{\tau}}Q_0\varphi)-\nabla_{w,\pmb{\tau}}Q_h\varphi\|_{\pa T}^2\\
=&\sum_{T\in{\cal T}_h}h_T^{-1}\|Q_n(\bn_{f}\times(\nabla Q_0\varphi\times\bn_{f}))-Q_n(\bn_{f}\times(\nabla\varphi\times\bn_{f}))\|_{\pa T}^2\\
\lesssim&\sum_{T\in{\cal T}_h}h_T^{-1}\|\bn_{f}\times(\nabla Q_0\varphi\times\bn_{\F})-\bn_{f}\times(\nabla\varphi\times\bn_{f})\|_{\pa T}^2\\
\lesssim&\sum_{T\in{\cal T}_h}h_T^{-1}\|\nabla Q_0\varphi-\nabla\varphi\|_{\pa T}^2\\
\lesssim&\ h^{k-1}\|\varphi\|_{k+1}.
\end{split}
\end{equation*}
 This completes the proof of the lemma.
 \end{proof}

\section{Error estimates}\label{Section:EE}
The following is an error estimate for the numerical scheme \eqref{WG-scheme} with respect to the natural ``energy" norm.

\begin{theorem}\label{THM:energy-estimate}
Let $u$ be the exact solution of the equation \eqref{model-problem}  and $u_h\in V_h$ be its numerical approximation arising from the WG scheme \eqref{WG-scheme}.  Under the assumtion of $u\in H^{k+1}(\O)$, the following error estimate holds true:
\begin{equation}\label{tribarerror}
\3bare_h\3bar\lesssim h^{k-1}\|u\|_{k+1}.
\end{equation}
\end{theorem}
\begin{proof}
By taking $v=e_h\in V_h^0$ in \eqref{Error-equation} we have
\begin{equation}\label{energy-estimate-2}
\begin{split}
\3bare_h\3bar^2
=&\zeta_u(e_h)\\
=&s(Q_hu,e_h)+\sum_{T\in{\cal T}_h}\sum_{i,j=1}^d\langle e_0-e_{b,f},\pa_j(\mathbb{Q}_h(\pa^2_{ij}u)-\pa^2_{ij}u) n_i\rangle_{\pa T}\\
&+\sum_{T\in{\cal T}_h}\sum_{i,j=1}^d\langle \pa_ie_0-e_{gi},(\pa^2_{ij}u-\mathbb{Q}_h(\pa^2_{ij}u)) n_j\rangle_{\pa T}\\
=&I_1+I_2+I_3.
\end{split}
\end{equation}

For $I_1$, we have from the Cauchy-Schwarz inequality that
\begin{equation*}
\begin{split}
 |I_1|=&\ |s(Q_hu,e_h)|\\
\leq&\sum_{T\in{\cal T}_h}h_T^{-2}|\langle Q_b(Q_0u)-Q_bu,Q_be_0-e_{b,e}\rangle_{\pa \F}|\\
  &+\sum_{T\in{\cal T}_h}h_T^{-3} |\langle Q_f(Q_0u)-Q_fu,Q_fe_0-e_{b,f}\rangle_{\pa T}|\\
  &+\sum_{T\in{\cal T}_h}h_T^{-1}|\langle Q_n(\nabla Q_0u)\cdot\bn_f-Q_n(\nabla u\cdot\bn_f),Q_n(\nabla e_0)\cdot\bn_f-e_n\rangle_{\pa T}|\\
  &+\delta_{k,3}\sum_{T\in{\cal T}_h}h_T^{-1}|\langle Q_nD_{\pmb{\tau}}Q_0u-\nabla_{w,\pmb{\tau}}Q_hu,Q_nD_{\pmb{\tau}}e_0-\nabla_{w,\pmb{\tau}}e_h\rangle_{\pa T}|\\
{ \lesssim}&\Big(\sum_{T\in{\cal T}_h}h_T^{-2}\|Q_0u-u\|_{\pa \F}^2\Big)^{\frac{1}{2}}
     \Big(\sum_{T\in{\cal T}_h}h_T^{-2}\|Q_be_0-e_{b,e}\|_{\pa \F}^2\Big)^{\frac{1}{2}}\\
&+\Big(\sum_{T\in{\cal T}_h}h_T^{-3}\|Q_0u-u\|_{\pa T}^2\Big)^{\frac{1}{2}}
     \Big(\sum_{T\in{\cal T}_h}h_T^{-3}\|Q_fe_0-e_{b,f}\|_{\pa T}^2\Big)^{\frac{1}{2}}\\
&+\Big(\sum_{T\in{\cal T}_h}h_T^{-1}\|\nabla Q_0u-\nabla u\|_{\pa T}^2\Big)^{\frac{1}{2}}
    \Big(\sum_{T\in {\cal T}_h}h_T^{-1}\|Q_n(\nabla e_0)\cdot\bn_f-e_n\|_{\pa T}^2\Big)^{\frac{1}{2}}\\
&+\delta_{k,3}\Big(\sum_{T\in{\cal T}_h}h_T^{-1}\|Q_nD_{\pmb{\tau}}Q_0u-\nabla_{w,\pmb{\tau}}Q_hu\|_{\pa T}^2\Big)^{\frac{1}{2}}\3bare_h\3bar.
\end{split}
\end{equation*}
Next, using the trace inequality \eqref{trace inequality} and the estimates \eqref{q0} and \eqref{esti2}, we arrive at
\begin{equation}\label{energy-estimate-3}
\begin{split}
|I_1|\lesssim&\Big(\sum_{T\in{\cal T}_h}h_T^{-3}\|Q_0u-u\|_{\F}^2+h_T^{-1}\|\nabla Q_0u-\nabla u\|_{\F}^2\Big)^{\frac{1}{2}}\3bare_h\3bar\\
&+\Big(\sum_{T\in{\cal T}_h}h_T^{-4}\|Q_0u-u\|_{T}^2+h_T^{-2}\|\nabla Q_0u-\nabla u\|_{T}^2\Big)^{\frac{1}{2}}\3bare_h\3bar\\
&+\Big(\sum_{T\in{\cal T}_h}h_T^{-2}\|\nabla Q_0u-\nabla u\|_{ T}^2+|Q_0u-u|_{2,T}^2\Big)^{\frac{1}{2}}\3bare_h\3bar\\
&+\delta_{k,3}h^2\|u\|_4\3bare_h\3bar\\
\lesssim&\Big(\sum_{T\in{\cal T}_h}h_T^{-4}\|Q_0u-u\|_{T}^2+h_T^{-2}\|\nabla Q_0u-\nabla u\|_{T}^2\Big)^{\frac{1}{2}}\3bare_h\3bar\\
&+\Big(\sum_{T\in{\cal T}_h}h_T^{-2}\|\nabla Q_0u-\nabla u\|_{T}^2+|Q_0u-u|_{2,T}^2\Big)^{\frac{1}{2}}\3bare_h\3bar\\
&+\delta_{k,3}h^2\|u\|_4\3bare_h\3bar\\
\lesssim&\ h^{k-1}\|u\|_{k+1}\3bare_h\3bar+\delta_{k,3}h^2\|u\|_4\3bare_h\3bar\\
\lesssim&\ h^{k-1}\|u\|_{k+1}\3bare_h\3bar.
\end{split}
\end{equation}

For the term $I_2$, we have from  \eqref{esti1}, the Cauchy-Schwarz inequality, and the trace inequality \eqref{trace inequality} that
\begin{equation}\label{energy-estimate-51}
\begin{split}
|I_2|=&|\sum_{T\in{\cal T}_h}\sum_{i,j=1}^d\langle e_0-e_{b,f},\pa_j(\mathbb{Q}_h(\pa^2_{ij}u)-\pa^2_{ij}u)n_i\rangle_{\pa T}|\\
=&|\sum_{T\in{\cal T}_h}\sum_{i,j=1}^d\langle e_0-Q_fe_{0},\pa_j(\mathbb{Q}_h(\pa^2_{ij}u)-\pa^2_{ij}u)n_i\rangle_{\pa T}\\
&+\sum_{T\in{\cal T}_h}\sum_{i,j=1}^d\langle Q_fe_{0}-e_{b,f},\pa_j(\mathbb{Q}_h(\pa^2_{ij}u)-\pa^2_{ij}u)n_i\rangle_{\pa T}|\\
\lesssim&\ h^{k-1}\|u\|_{k+1}\3bar e_h\3bar+\Big(\sum_{T\in{\cal T}_h}h_T^{-3}\|Q_fe_0-e_{b,f}\|_{\pa T}^2\Big)^{\frac{1}{2}} \\
&\cdot\Big(\sum_{T\in{\cal T}_h}\sum_{i,j=1}^dh_T^3\|\pa_j(\mathbb{Q}_h(\pa^2_{ij}u)-\pa^2_{ij}u)\|_{\pa T}^2\Big)^{\frac{1}{2}}\\
\lesssim&\ h^{k-1}\|u\|_{k+1}\3bar e_h\3bar.
\end{split}
\end{equation}

As to $I_3$, we have from the Cauchy-Schwarz inequality, Lemmas \ref{PROJECTION-ESTIMATE-68}-\ref{PROJECTION-ESTIMATE-6},  the trace inequality \eqref{trace inequality},  and \eqref{qh}  that
\begin{equation}\label{energy-estimate-4}
\begin{split}
|I_3|=&|\sum_{T\in{\cal T}_h}\sum_{i,j=1}^d\langle \pa_ie_0-e_{gi},(\pa^2_{ij}u-\mathbb{Q}_h(\pa^2_{ij}u)) n_j\rangle_{\pa T}|\\
\lesssim&\Big(\sum_{T\in{\cal T}_h}\sum_{i=1}^dh_T^{-1}\|Q_n(\pa_ie_0)-e_{gi}\|_{\pa T}^2
 +h_T^{-1}\|\pa_ie_0-Q_n(\pa_ie_0)\|_{\pa T}^2\Big)^{\frac{1}{2}}\\
&\cdot\Big(\sum_{T\in{\cal T}_h}\sum_{i,j=1}^dh_T\|\pa^2_{ij}u-\mathbb{Q}_h(\pa^2_{ij}u)\|_{\pa T}^2\Big)^{\frac{1}{2}}\\
\lesssim&\Big(\3bare_h\3bar^2+\sum_{T\in{\cal T}_h}|\pa_ie_0|_{1,T}^2\Big)^{\frac{1}{2}}\\
&\cdot\Big(\sum_{T\in{\cal T}_h}\sum_{i,j=1}^d\|\pa^2_{ij}u-\mathbb{Q}_h(\pa^2_{ij}u)\|_{T}^2
 +h_T^2\|\nabla(\pa^2_{ij}u-\mathbb{Q}_h(\pa^2_{ij}u))\|_{T}^2\Big)^{\frac{1}{2}}\\
\lesssim&\Big(\3bare_h\3bar^2+\sum_{T\in{\cal T}_h}|e_0|_{2,T}^2\Big)^{\frac{1}{2}}h^{k-1}\|u\|_{k+1}\\
\lesssim&\ h^{k-1}\|u\|_{k+1}\3bare_h\3bar.
\end{split}
\end{equation}
Substituting \eqref{energy-estimate-3}-\eqref{energy-estimate-4} into \eqref{energy-estimate-2} gives rise to \eqref{tribarerror}. This completes the proof of the theorem.
\end{proof}

To establish an optimal order error estimate for the numerical solution in the $L^2$ norm, we consider the dual problem that seeks $\Phi$ satisfying
\begin{equation}\label{dual-equation}
\begin{split}
\Delta^2\Phi&=e_0,\quad\mbox{in}~~\O,\\
        \Phi&=0,\ \quad\mbox{on}~~\pa\O,\\
\frac{\pa\Phi}{\pa\textbf{n}}&=0,\ \quad\mbox{on}~~\pa\O.
\end{split}
\end{equation}
Assume that the problem \eqref{dual-equation} has the $H^4$-regularity in the sense that there exists a constant $C$ such that
\begin{equation}\label{dual-regular}
\|\Phi\|_4\leq C\|e_0\|.
\end{equation}

\begin{theorem}\label{THM:L2-estimate-e0}
Let $u\in H^{k+1}(\O)$ be the exact solution of the problem \eqref{model-problem}  and $u_h\in V_h$ be its numerical solution arising from the WG scheme \eqref{WG-scheme}. Under the $H^4$-regularity assumption \eqref{dual-regular}, we have the following error estimate
$$
\|e_0\|\lesssim h^{k+1}\|u\|_{k+1}.
$$
\end{theorem}
\begin{proof}
First, using \eqref{discrete weak gradient-1} with $e_{b,f}=0$ on each $\F\subset\pa\O$ and $e_{b,e}=0$ on each $e\subset\pa\O$ gives $\nabla_{w,\pmb{\tau}}e_h=0$ on each $\F\subset\pa\O$. This, together with $e_n=0$ on $\pa\O$  and \eqref{decomposition}, gives $\be_g=0$ on $\pa\O$. Next, we test the dual problem \eqref{dual-equation} against $e_0$ and use the integration by parts to obtain
\begin{equation}\label{THM:L2-estimate-1}
\begin{split}
\|e_0\|^2
=&(\Delta^2\Phi,e_0)\\
=&\sum_{T\in{\cal T}_h}\sum_{i,j=1}^d(\pa^2_{ij}\Phi,\pa^2_{ij}e_0)_T-\langle\pa^2_{ij}\Phi,\pa_ie_0 n_j\rangle_{\pa T}
  +\langle\pa_j(\pa^2_{ij}\Phi) n_i,e_0\rangle_{\pa T}\\
=&\sum_{T\in{\cal T}_h}\sum_{i,j=1}^d(\pa^2_{ij}\Phi,\pa^2_{ij}e_0)_T-\langle\pa^2_{ij}\Phi,(\pa_ie_0-e_{gi}) n_j\rangle_{\pa T}\\
&+\langle\pa_j(\pa^2_{ij}\Phi) n_i,e_0-e_{b,f}\rangle_{\pa T},
\end{split}
\end{equation}
 where we have used $\sum_{T\in{\cal T}_h} \langle\pa^2_{ij}\Phi, e_{gi} n_j\rangle_{\pa T}=0$ and $\sum_{T\in{\cal T}_h}\langle\pa_j(\pa^2_{ij}\Phi) n_i, e_{b,f}\rangle_{\pa T}=0$ since {$e_{b,f}=0$ and} $\be_{g}=0$ on $\pa\O$.

Analogues to \eqref{error equation-3}, we have
\begin{equation*}
\begin{split}
(\pa_{ij}^2\Phi,\pa_{ij}^2e_0)_T
=&(\pa_{ij,w}^2e_h,\pa_{ij,w}^2Q_h\Phi)_T+\langle (e_{b,f}-e_0)n_i,\pa_j(\mathbb{Q}_h(\pa^2_{ij}\Phi))\rangle_{\pa T}\\
 &+\langle\pa_ie_0-e_{gi},\mathbb{Q}_h(\pa^2_{ij}\Phi) n_j\rangle_{\pa T},
\end{split}
\end{equation*}
which, together with \eqref{THM:L2-estimate-1} and \eqref{Error-equation}-\eqref{error equation-remainder}, leads to
\begin{equation}\label{THM:L2-estimate-3}
\begin{split}
 \|e_0\|^2
=&\sum_{T\in{\cal T}_h}\sum_{i,j=1}^d(\pa_{ij,w}^2e_h,\pa_{ij,w}^2Q_h\Phi)_T
+\langle(\pa_ie_0-e_{gi}) n_j,\mathbb{Q}_h(\pa^2_{ij}\Phi)-\pa^2_{ij}\Phi\rangle_{\pa T}\\
&+\langle(e_0-e_{b,f}) n_i,\pa_j(\pa^2_{ij}\Phi-\mathbb{Q}_h(\pa^2_{ij}\Phi))\rangle_{\pa T}\\
=&\zeta_u(Q_h\Phi)-\zeta_\Phi(e_h)\\
=& \sum_{i=1}^3 J_i-\zeta_\Phi(e_h),
\end{split}
\end{equation}
where $J_i$ are given as in \eqref{error equation-remainder} with $v=Q_h\Phi.$

The rest of the proof amounts to the estimate for each of the four terms on the last line in  \eqref{THM:L2-estimate-3}.

For $J_1$, we have from Cauchy-Schwarz inequality, \eqref{esti2}, the trace inequality \eqref{trace inequality}, \eqref{q0}, and the $H^4$ regularity assumption \eqref{dual-regular} that
\begin{equation}\label{THM:L2-estimate-5}
\begin{split}
 &|J_1|\\
=&|\sum_{T\in{\cal T}_h}h_T^{-2}\langle Q_b(Q_0u)-Q_bu,Q_b(Q_0\Phi)-Q_b\Phi\rangle_{\pa \F}\\
 &+h_T^{-3}\langle Q_f(Q_0u)-Q_fu,Q_f(Q_0\Phi)-Q_f\Phi\rangle_{\pa T}\\
 &+h_T^{-1}\langle Q_n(\nabla Q_0u)\cdot\bn_f-Q_n(\nabla u\cdot\bn_f),Q_n(\nabla Q_0\Phi)\cdot\bn_f-Q_n(\nabla\Phi\cdot\bn_f)\rangle_{\pa T}\\
 &+\delta_{k,3}\sum_{T\in{\cal T}_h}h_T^{-1}\langle Q_nD_{\pmb{\tau}}Q_0u-\nabla_{w,\pmb{\tau}}Q_hu,Q_nD_{\pmb{\tau}}Q_0\Phi-\nabla_{w,\pmb{\tau}}Q_h\Phi\rangle_{\pa T}|\\
\lesssim&\Big(\sum_{T\in{\cal T}_h}h_T^{-2}\|Q_0u-u\|_{\pa \F}^2\Big)^{\frac{1}{2}}
  \Big(\sum_{T\in{\cal T}_h}h_T^{-2}\|Q_0\Phi-\Phi\|_{\pa \F}^2\Big)^{\frac{1}{2}}\\
&+\Big(\sum_{T\in{\cal T}_h}h_T^{-3}\|Q_0u-u\|_{\pa T}^2\Big)^{\frac{1}{2}}
  \Big(\sum_{T\in{\cal T}_h}h_T^{-3}\|Q_0\Phi-\Phi\|_{\pa T}^2\Big)^{\frac{1}{2}}\\
&+\Big(\sum_{T\in{\cal T}_h}h_T^{-1}\|\nabla Q_0u-\nabla u\|_{\pa T}^2\Big)^{\frac{1}{2}}
  \Big(\sum_{T\in{\cal T}_h}h_T^{-1}\|\nabla Q_0\Phi-\nabla\Phi\|_{\pa T}^2\Big)^{\frac{1}{2}}\\
  &+\delta_{k,3}h^4\|u\|_4 \|\Phi\|_4\\
\lesssim&\ h^{k+1}\|u\|_{k+1}\|\Phi\|_4\\
\lesssim&\ h^{k+1}\|u\|_{k+1}\|e_0\|.
\end{split}
\end{equation}

For the term $J_2$, we have
\begin{equation*}
\begin{split}
 J_2=& \sum_{T\in{\cal T}_h}\sum_{i,j=1}^d\langle Q_0\Phi-Q_f\Phi,\pa_j(\mathbb{Q}_h(\pa^2_{ij}u)-\pa^2_{ij}u) n_i\rangle_{\pa T}\\
=&\sum_{T\in{\cal T}_h}\sum_{i,j=1}^d\langle (Q_0\Phi-\Phi)+(\Phi-Q_f\Phi),\pa_j(\mathbb{Q}_h(\pa^2_{ij}u)-\pa^2_{ij}u) n_i\rangle_{\pa T}\\
=&\sum_{T\in{\cal T}_h}\sum_{i,j=1}^d\langle Q_0\Phi-\Phi,\pa_j((\mathbb{Q}_h-I)\pa^2_{ij}u) n_i\rangle_{\pa T}
 \\&+\langle \Phi-Q_f\Phi,\pa_j(\pa^2_{ij}u)n_i\rangle_{\pa T}\\
=&\sum_{T\in{\cal T}_h}\sum_{i,j=1}^d\langle Q_0\Phi-\Phi,\pa_j(\mathbb{Q}_h\pa^2_{ij}u-\pa^2_{ij}u) n_i\rangle_{\pa T},
\end{split}
\end{equation*}
where we used the fact that $\sum_{T\in\T_h}\langle \Phi-Q_f\Phi,\pa_j(\pa^2_{ij}u) n_i\rangle_{\pa T}=0$.
It follows that
\begin{equation}\label{THM:L2-estimate-6}
\begin{split}
 |J_2| \lesssim&\Big(\sum_{T\in{\cal T}_h}\|Q_0\Phi-\Phi\|_{\pa T}^2\Big)^{\frac{1}{2}}
           \Big(\sum_{T\in{\cal T}_h}\sum_{i,j=1}^d\|\pa_j(\mathbb{Q}_h(\pa^2_{ij}u)-\pa^2_{ij}u)\|_{\pa T}^2\Big)^{\frac{1}{2}}\\
\lesssim&\ h^{k+1}\|\Phi\|_4\ \|u\|_{k+1}\\
\lesssim&\ h^{k+1}\|u\|_{k+1}{\|e_0\|.}
\end{split}
\end{equation}

For the term $J_3$, we note that the weak gradient of the $L^2$ projection of a smooth function is the same as the $L^2$ projection of its classical gradient on the boundary of each element,  see \eqref{decomposition-new}. Hence,
$$
J_3=\sum_{T\in{\cal T}_h}\sum_{i,j=1}^d\langle\pa_iQ_0\Phi-Q_n(\pa_i\Phi),
 (\pa^2_{ij}u-\mathbb{Q}_h(\pa^2_{ij}u))n_j\rangle_{\pa T}.
$$
It follows from the Cauchy-Schwarz inequality, the trace inequality \eqref{trace inequality}, Lemma \ref{error projection}, and the regularity assumption \eqref{dual-regular} that
\begin{equation}\label{THM:L2-estimate-7}
\begin{split}
 |J_3|\lesssim&\Big(\sum_{T\in{\cal T}_h}h_T^{-1}\|\nabla Q_0\Phi-\nabla\Phi\|_{T}^2+h_T|\nabla Q_0\Phi-\nabla\Phi|_{1, T}^2\Big)^{\frac{1}{2}}\\
&\cdot\Big(\sum_{T\in{\cal T}_h}\sum_{i,j=1}^dh_T^{-1}\|\pa^2_{ij}u-\mathbb{Q}_h(\pa^2_{ij}u)\|_{T}^2
 +h_T\|\nabla(\pa^2_{ij}u-\mathbb{Q}_h(\pa^2_{ij}u))\|_{T}^2\Big)^{\frac{1}{2}}\\
\lesssim&\ h^{k+1}\|\Phi\|_4\ \|u\|_{k+1}\\
\lesssim&\ h^{k+1}\|u\|_{k+1}\|e_0\|.
\end{split}
\end{equation}

To deal with the last term, using the same arguments as in \eqref{energy-estimate-3}-\eqref{energy-estimate-4} with $u=\Phi$ and then combining  \eqref{tribarerror} with \eqref{dual-regular}, there yields
\begin{equation}\label{THM:L2-estimate-4}
\begin{split}
|\zeta_\Phi(e_h)|
\lesssim&h^2\|\Phi\|_4\3bare_h\3bar\\
\lesssim&h^{k+1}\|u\|_{k+1}\|\Phi\|_4\\
\lesssim&h^{k+1}\|u\|_{k+1}\|e_0\|.
\end{split}
\end{equation}

Finally, substituting \eqref{THM:L2-estimate-5}-\eqref{THM:L2-estimate-4} into \eqref{THM:L2-estimate-3}  completes the proof of the theorem.
\end{proof}

We further introduce the following measure for the numerical solutions on element boundaries:
\begin{equation*}
\begin{split}
 \|e_{b,e}\|_{\E_h}=&\Big(\sum_{T\in{\cal T}_h}h_T^2\|e_{b,e}\|_{\pa \F}^2\Big)^{\frac{1}{2}},\\
 \|e_{b,f}\|_{{\F}_h}=&\Big(\sum_{T\in{\cal T}_h}h_T\|e_{b,f}\|_{\pa T}^2\Big)^{\frac{1}{2}},\\
 \|e_n\|_{{\F}_h}=&\Big(\sum_{T\in{\cal T}_h}h_T\|e_n\|_{\pa T}^2\Big)^{\frac{1}{2}}.
\end{split}
\end{equation*}

\begin{theorem}\label{THM:L2-estimate-edge}
Under the assumptions of Theorem \ref{THM:L2-estimate-e0}, there holds
\begin{eqnarray}
\|e_{b,e}\|_{\E_h}&\lesssim & h^{k+1}\|u\|_{k+1},\label{errebe}
\\ \label{errebf}
\|e_{b,f}\|_{{\F}_h}&\lesssim & h^{k+1}\|u\|_{k+1},
\\ \label{erreben}
\|e_n\|_{{\F}_h}&\lesssim & h^k\|u\|_{k+1}.
\end{eqnarray}
\end{theorem}

\begin{proof}
From the triangular inequality, the trace inequality \eqref{inverse inequality}, \eqref{tri-semibar}, Theorems \ref{THM:energy-estimate} and \ref{THM:L2-estimate-e0}, there holds
\begin{equation*}\label{L2-estimate-edge-2}
\begin{split}
\|e_{b,e}\|_{\E_h}
=&\Big(\sum_{T\in{\cal T}_h}h_T^2\|e_{b,e}\|_{\pa \F}^2\Big)^{\frac{1}{2}}\\
\lesssim&\Big(\sum_{T\in{\cal T}_h}h_T^2\|Q_be_0\|_{\pa \F}^2+h_T^2\|e_{b,e}-Q_be_0\|_{\pa \F}^2\Big)^{\frac{1}{2}}\\
\lesssim&\Big(\sum_{T\in{\cal T}_h}h_T^2h_T^{-1}\|e_0\|_{\pa T}^2+h_T^2h_T^2\3bare_h\3bar^2\Big)^{\frac{1}{2}}\\
\lesssim&\Big(\sum_{T\in{\cal T}_h}h_Th_T^{-1}\|e_0\|_T^2+h_T^4h_T^{2(k-1)}\|u\|_{k+1}^2\Big)^{\frac{1}{2}}\\
\lesssim&h^{k+1}\|u\|_{k+1},
\end{split}
\end{equation*}
which completes the proof for \eqref{errebe}.

The proof for \eqref{errebf} and \eqref{erreben} can be obtained by using a similar argument.

\end{proof}

\section{Numerical experiments}\label{Section:NE}

In this section, the numerical scheme \eqref{WG-scheme} will be implemented to verify the convergence theory established in the previous sections. To this end, we first solve the biharmonic equation \eqref{model-problem} on the unit square $\Omega=(0,1)^2$, where $g$ and the boundary conditions are chosen so that the exact
 solution is
\begin{equation}
    u(x,y)=2^8(x-x^2)^2(y-y^2)^2. \label{sol}
\end{equation}

\textbf{Test Example 1.}  We take the square as the initial mesh, and subdivide each square into four to get subsequent meshes, as shown in Figure \ref{grid1}. One can see from Table \ref{t1} that the optimal rates of convergence are obtained in the usual $L^2$ and $H^2$-like {triple-bar norm} for $P_3$,  $P_4$ and $P_5$ WG methods.

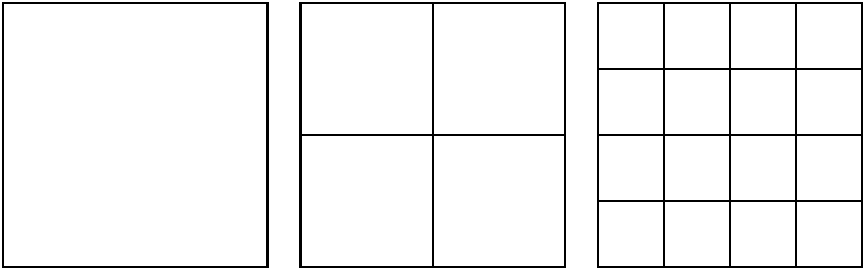
\begin{figure}[ht!]
 \begin{center} \setlength\unitlength{1.25pt}
\begin{picture}(260,80)(0,0)
  \def\tr{\begin{picture}(20,20)(0,0)\put(0,0){\line(1,0){20}}\put(0,20){\line(1,0){20}}
          \put(0,0){\line(0,1){20}} \put(20,0){\line(0,1){20}}
      \end{picture}}
 {\setlength\unitlength{5pt}
 \multiput(0,0)(20,0){1}{\multiput(0,0)(0,20){1}{\tr}}}

  {\setlength\unitlength{2.5pt}
 \multiput(45,0)(20,0){2}{\multiput(0,0)(0,20){2}{\tr}}}

  \multiput(180,0)(20,0){4}{\multiput(0,0)(0,20){4}{\tr}}

 \end{picture}\end{center}
\caption{The first three levels of square grids used in Table \ref{t1} computation.}
\label{grid1}
\end{figure}
\begin{table}[ht!]
    \caption{The error profile
  for solving \eqref{sol} on square grids shown in Figure \ref{grid1}.}
    \label{t1}
    \begin{center}
    \begin{tabular}{|c|cc|cc|}  
        \hline
Grid& $\|Q_hu-u_h\| $ & Rate &$\3bar Q_hu-u_h \3bar$&Rate   \\
        \hline
   &\multicolumn{4}{c|} { The $P_3$ WG finite element }  \\
        \hline
 5&   0.1486E-02 & 3.90&   0.9339E+00 & 1.95 \\
 6&   0.9595E-04 & 3.95&   0.2373E+00 & 1.98 \\
 7&   0.6092E-05 & 3.98&   0.5981E-01 & 1.99 \\
\hline
   &\multicolumn{4}{c|} { The $P_4$ WG finite element }  \\
        \hline
 3&   0.3791E-01 & 3.85&   0.3692E+01 & 2.86 \\
 4&   0.1330E-02 & 4.83&   0.4803E+00 & 2.94 \\
 5&   0.4232E-04 & 4.97&   0.6068E-01 & 2.98 \\
\hline
   &\multicolumn{4}{c|} { The $P_5$ WG finite element }  \\
        \hline
 2&   0.2460E+00 & 5.05&   0.1823E+02 & 5.21 \\
 3&   0.5110E-02 & 5.59&   0.9983E+00 & 4.19 \\
 4&   0.8558E-04 & 5.90&   0.5589E-01 & 4.16 \\
\hline
\end{tabular}
\end{center}
\end{table}

\textbf{Test Example 2.}  We take the uniform triangular meshes, as shown in Figure \ref{grid2}. One can see from Table \ref{t2} that optimal rates of convergence are demonstrated in the usual $L^2$ and $H^2$-like
{triple-bar norm} for $P_3$, $P_4$ and $P_5$ WG methods.

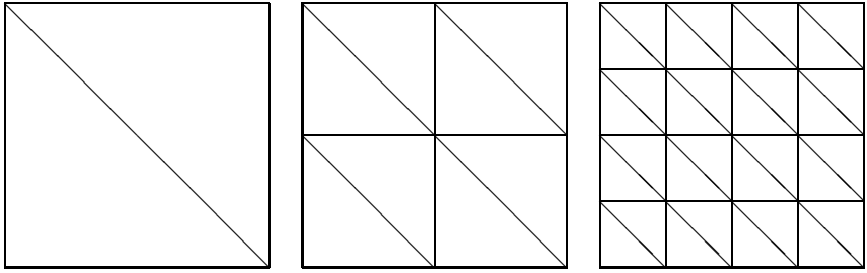
\begin{figure}[ht!]
 \begin{center} \setlength\unitlength{1.25pt}
\begin{picture}(260,80)(0,0)
  \def\tr{\begin{picture}(20,20)(0,0)\put(0,0){\line(1,0){20}}\put(0,20){\line(1,0){20}}
          \put(0,0){\line(0,1){20}} \put(20,0){\line(0,1){20}}
        \put(0,20){\line(1,-1){20}}
      \end{picture}}
 {\setlength\unitlength{5pt}
 \multiput(0,0)(20,0){1}{\multiput(0,0)(0,20){1}{\tr}}}

  {\setlength\unitlength{2.5pt}
 \multiput(45,0)(20,0){2}{\multiput(0,0)(0,20){2}{\tr}}}

  \multiput(180,0)(20,0){4}{\multiput(0,0)(0,20){4}{\tr}}

 \end{picture}\end{center}
\caption{The first three levels of triangular grids used in Table \ref{t2} computation.}
\label{grid2}
\end{figure}
\begin{table}[ht!]
    \caption{The error profile for solving \eqref{sol} on triangular grids shown in Figure \ref{grid2}.}
    \label{t2}
    \begin{center}
    \begin{tabular}{|c|cc|cc|}  
        \hline
Grid& $\|Q_hu-u_h\| $ & Rate &$\3bar Q_hu-u_h \3bar$&Rate   \\
        \hline
   &\multicolumn{4}{c|} { The $P_3$ WG finite element }  \\
        \hline
 5&   0.8263E-03 & 3.97&   0.7030E+00 & 1.98 \\
 6&   0.5190E-04 & 3.99&   0.1764E+00 & 2.00 \\
 7&   0.3252E-05 & 4.00&   0.4414E-01 & 2.00 \\
\hline
   &\multicolumn{4}{c|} { The $P_4$ WG finite element }  \\
        \hline
 4&   0.6526E-03 & 4.86&   0.2874E+00 & 2.88 \\
 5&   0.2088E-04 & 4.97&   0.3666E-01 & 2.97 \\
 6&   0.6563E-06 & 4.99&   0.4606E-02 & 2.99 \\
\hline
   &\multicolumn{4}{c|} { The $P_5$ WG finite element }  \\
        \hline
 3&   0.2622E-02 & 5.59&   0.3941E+00 & 3.65 \\
 4&   0.4362E-04 & 5.91&   0.2601E-01 & 3.92 \\
 5&   0.6929E-06 & 5.98&   0.1649E-02 & 3.98 \\
\hline
\end{tabular}
\end{center}
\end{table}

\textbf{Test Example 3. } We take polygonal meshes shown as in Figure \ref{grid3}. Table \ref{t3} illustrates the corresponding numerical results which clearly demonstrate optimal rates of convergence in the usual $L^2$ and $H^2$-like triple-bar norms for $P_3$,  $P_4$ and $P_5$ WG methods.

\begin{figure}[ht!]
 \begin{center} \setlength\unitlength{1.25pt}
\begin{picture}(180,80)(0,0)
  \def\tr{\begin{picture}(20,20)(0,0)\put(0,0){\line(1,0){20}}\put(0,20){\line(1,0){20}}
          \put(0,0){\line(0,1){20}} \put(20,0){\line(0,1){20}}
        \put(20,0){\line(-3,2){6}}\put(0,0){\line(3,2){6}}\put(6,4){\line(1,0){8}}
        \put(20,20){\line(-3,-2){6}}\put(0,20){\line(3,-2){6}}\put(6,16){\line(1,0){8}}
  \put(4,10){\line(1,3){2}}\put(4,10){\line(1,-3){2}}
  \put(16,10){\line(-1,3){2}}\put(16,10){\line(-1,-3){2}}
      \end{picture}}
 {\setlength\unitlength{5pt}
 \multiput(0,0)(20,0){1}{\multiput(0,0)(0,20){1}{\tr}}}

  {\setlength\unitlength{2.5pt}
 \multiput(45,0)(20,0){2}{\multiput(0,0)(0,20){2}{\tr}}}

 \end{picture}\end{center}
\caption{The first two levels of quadrilateral-pentagon-hexagon grids used in Table \ref{t3} computation.}
\label{grid3}
\end{figure}
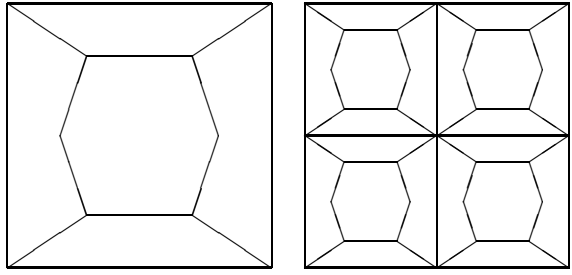

\begin{table}[ht!]
    \caption{The error profile for solving \eqref{sol} on polygonal grids shown in Figure \ref{grid3}.}
    \label{t3}
    \begin{center}
    \begin{tabular}{|c|cc|cc|}  
        \hline
Grid& $\|Q_hu-u_h\| $ & Rate &$\3bar Q_hu-u_h \3bar$&Rate   \\
        \hline
   &\multicolumn{4}{c|} { The $P_3$ WG finite element }  \\
        \hline
 4&   0.5052E-02 & 3.92&   0.1724E+01 & 1.94 \\
 5&   0.3207E-03 & 3.98&   0.4355E+00 & 1.98 \\
 6&   0.1999E-04 & 4.00&   0.1092E+00 & 2.00 \\
\hline
   &\multicolumn{4}{c|} { The $P_4$ WG finite element }  \\
        \hline
 3&   0.4732E-02 & 4.96&   0.9861E+00 & 3.04 \\
 4&   0.1473E-03 & 5.01&   0.1213E+00 & 3.02 \\
 5&   0.4974E-05 & 4.89&   0.1510E-01 & 3.01 \\
\hline
   &\multicolumn{4}{c|} { The $P_5$ WG finite element }  \\
        \hline
 1&   0.1201E+01 & 0.00&   0.3857E+02 & 0.00 \\
 2&   0.1468E-01 & 6.36&   0.1683E+01 & 4.52 \\
 3&   0.2705E-03 & 5.76&   0.9122E-01 & 4.21 \\
\hline
\end{tabular}
\end{center}
\end{table}

\textbf{Test Example 4.} We solve the biharmonic equation \eqref{model-problem} on the unit cubic domain $\Omega=(0,1)^3$, where $g$ and the boundary conditions are chosen so that the exact solution is given by
\begin{equation}
    u(x,y,z)=2^{12}(x-x^2)^2(y-y^2)^2(z-z^2)^2. \label{sol2}
\end{equation}
In this test, we use the uniform cube meshes shown as in Figure \ref{grid4}.
The results from the $P_3$ and $P_4$ WG methods are shown in Table \ref{t4}.
The optimal order of convergence is achieved in all cases.

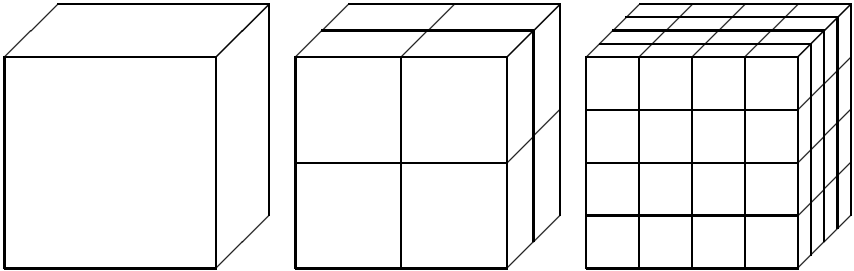
\begin{figure}[ht!]
\begin{center}
 \setlength\unitlength{1pt}
    \begin{picture}(320,118)(0,3)
    \put(0,0){\begin{picture}(110,110)(0,0)
       \multiput(0,0)(80,0){2}{\line(0,1){80}}  \multiput(0,0)(0,80){2}{\line(1,0){80}}
       \multiput(0,80)(80,0){2}{\line(1,1){20}} \multiput(0,80)(20,20){2}{\line(1,0){80}}
       \multiput(80,0)(0,80){2}{\line(1,1){20}}  \multiput(80,0)(20,20){2}{\line(0,1){80}}
      \end{picture}}
    \put(110,0){\begin{picture}(110,110)(0,0)
       \multiput(0,0)(40,0){3}{\line(0,1){80}}  \multiput(0,0)(0,40){3}{\line(1,0){80}}
       \multiput(0,80)(40,0){3}{\line(1,1){20}} \multiput(0,80)(10,10){3}{\line(1,0){80}}
       \multiput(80,0)(0,40){3}{\line(1,1){20}}  \multiput(80,0)(10,10){3}{\line(0,1){80}}
      \end{picture}}
    \put(220,0){\begin{picture}(110,110)(0,0)
       \multiput(0,0)(20,0){5}{\line(0,1){80}}  \multiput(0,0)(0,20){5}{\line(1,0){80}}
       \multiput(0,80)(20,0){5}{\line(1,1){20}} \multiput(0,80)(5,5){5}{\line(1,0){80}}
       \multiput(80,0)(0,20){5}{\line(1,1){20}}  \multiput(80,0)(5,5){5}{\line(0,1){80}}

      \end{picture}}

    \end{picture}
    \end{center}
\caption{  The first three levels of cube grids used in the computation of Table \ref{t4}. }
\label{grid4}
\end{figure}

\begin{table}[ht!]
    \caption{The error profile for solving \eqref{sol2} on cube grids shown in Figure \ref{grid4}.}
    \label{t4}
    \begin{center}
    \begin{tabular}{|c|cc|cc|}  
        \hline
Grid& $\|Q_hu-u_h\|$ & Rate &$\3bar Q_hu-u_h \3bar$&Rate   \\
        \hline
   &\multicolumn{4}{c|} { The $P_3$ WG finite element }  \\
        \hline
 2&  0.8474E-01&  6.2&  0.1633E+01&  4.5\\
 3&  0.2583E-02&  5.0&  0.1846E+00&  3.1\\
 4&  0.2063E-03&  3.6&  0.4861E-01&  1.9\\
\hline
   &\multicolumn{4}{c|} { The $P_4$ WG finite element }  \\
        \hline
 2&  0.2247E-01&  9.8&  0.1373E+01&  6.9 \\
 3&  0.4988E-03&  5.5&  0.1079E+00&  3.7 \\
 4&  0.1705E-04&  4.9&  0.1009E-01&  3.4 \\
\hline
\end{tabular}
\end{center}
\end{table}


\newpage

\begin{thebibliography}{99}

\bibitem{AMV2018} {\sc P. F. Antonietti, G. Manzini and M. Verani}, {\em
The fully nonconforming virtual element method for biharmonic problems}, Math. Model. Methods. Appl. Sci., vol. 28 (2), pp. 387-407, 2018.

\bibitem{BGZ2020} {\sc J. Burkardt, M. Gunzburger and W. Zhao}, {\em
High-precision computation of the weak Galerkin methods for the fourth-order problem}, Numer. Algor., vol. 84, pp. 181-205, 2020.

\bibitem{CDG2009} {\sc B. Cockburn, B. Dong and J. Guzm\'{a}n}, {\em
A hybridizable and superconvergent discontinuous Galerkin method for biharmonic problems}, J. Sci. Comput., vol. 40, pp. 141-187, 2009.

\bibitem{CH2020} {\sc L. Chen and X. Huang}, {\em
Nonconforming virtual element method for $2m$th order partial differential equations in $\mathbb{R}^n$}, Math. Comput., vol. 89 (324), pp. 1711-1744, 2020.

\bibitem{DR2022} {\sc Z. Dong and A. Ren}, {\em
Hybrid high-order and weak Galerkin methods for the biharmonic problem}, SIAM J. Numer. Anal., {vol. 60 (5),} pp. 2626-2656, 2022.

\bibitem{CHL2001} {\sc J. Huang, L. Li and J. Chen}, {\em
On mortar-type Morley element method for plate bending problem}, Appl. Numer. Math., vol. 37, pp. 519-533, 2001.

\bibitem{HH2011} {\sc J. Huang and X. Huang}, {\em
Local and parallel algorithms for fourth order problems discretized by the Morley-Wang-Xu element method}, Numer. Math., vol. 119, pp. 667-697, 2011.

\bibitem{HI2023} {\sc H. Ishizaka}, {\em
Morley finite element analysis for fourth-order elliptic equations under a semi-regular mesh condition}, https://arxiv.org/pdf/2302.08719.pdf.

 \bibitem{WWL-2022} {\sc D. Li, C. Wang and J. Wang}, {\em
Weak Galerkin methods based Morley elements on general polytopal partitions}, https://arxiv.org/pdf/2210.17518v1.pdf.

 \bibitem{WWL-2022-bihar} {\sc D. Li, C. Wang and J. Wang}, {\em
Generalized weak Galerkin finite element methods for biharmonic equations}, {J. Comput. Appl. Math.,} vol. 434, {pp. 115353,} 2023.

\bibitem{Eff-MWY2017} {\sc L. Mu, J. Wang and X. Ye}, {\em
Effective implementation of the weak Galerkin finite element methods for the biharmonic equation}, Comput. Math. Appl., vol. 74, pp. 1215-1222, 2017.

\bibitem{MWY2014} {\sc L. Mu, J. Wang and X. Ye}, {\em
Weak Galerkin finite element methods for the biharmonic equation on polytopal meshes}, Numer. Methods Partial Differ. Equ., vol. 30 (3), pp.  1003-1029, 2014.

 \bibitem{WYWM2013} {\sc L. Mu, J. Wang, Y. Wang and X. Ye}, {\em
A weak Galerkin mixed finite element method for biharmonic equations}, Numerical Solution of Partial Differential Equations: Theory, Algorithms, and Their Applications, Springer Proceedings in Mathematical Statistics., vol. 45, pp. 247-277, 2013.

\bibitem{MSB2007} {\sc I. Mozolevski, E. S\"{u}li and P. R. B\"{o}sing}, {\em
Hp-version a priori error analysis of interior penalty discontinuous Galerkin finite element approximations to the biharmonic equation}, J. Sci. Comput., vol. 30 (3), pp. 465-491, 2007.

\bibitem{MC2006} {\sc S. Mao and S. Chen}, {\em
Convergence analysis of Morley element on anisotropic meshes}, J. Comput. Math., vol. 24, pp. 169-180, 2006.

\bibitem{PS2013} {\sc C. Park and D. Sheen}, {\em
A quadrilateral Morley element for biharmonic equations}, Numer. Math., vol. 124, pp. 395-413, 2013.

\bibitem{rusa1988} {\sc V. Ruas}, {\em
A quadratic finite element method for solving biharmonic problems in $\mathbb{R}^n$}, Numer. Math., vol. 52, pp. 33-43, 1988.

\bibitem{RS2002} {\sc R. Stevenson}, {\em
An analysis of nonconforming multi-grid methods, leading to an improved method for the Morley element}, Math. Comp., vol. 72, pp. 55-81, 2002.

\bibitem{SX1998} {\sc Z. Shi and Z. Xie}, {\em
Multigrid methods for Morley element on nonnested meshes}, J. Comput. Math., vol. 16 (5), pp. 385-394, 1998.

\bibitem{wy} {\sc J. Wang and X. Ye}, {\em
A weak Galerkin finite element method for second-order elliptic problems}, J. Comput. Appl. Math., vol. 241, pp.103-115, 2013.

\bibitem{WW_bihar-2014} {\sc C. Wang and J. Wang}, {\em
An efficient numerical scheme for the biharmonic equation by weak Galerkin finite element methods on polygonal or polyhedral meshes}, Comput. Math. Appl., vol. 68, pp. 2314-2330, 2014.

\bibitem{WW_HWG-2015} {\sc C. Wang and J. Wang}, {\em
A hybridized weak Galerkin finite element method for the biharmonic equation}, Int. J. Numer. Anal. Model., vol. 12, pp. 302-317, 2015.

\bibitem{WY-ellip_MC2014} {\sc J. Wang and X. Ye}, {\em
A weak Galerkin mixed finite element method for second-order elliptic problems}, Math. Comp., vol. 83, pp. 2101-2126, 2014.

\bibitem{WWYZ-JCAM2021} {\sc C. Wang, J. Wang, X. Ye and S. Zhang}, {\em
De Rham complexes for weak Galerkin finite element spaces}, J. Comput. Appl. Math., vol. 397, pp. 113645, 2021.

\bibitem{wwfp} {C. Wang and J. Wang}, {\em
A primal-dual weak Galerkin finite element method for Fokker-Planck type equations}, SIAM. J. Numer. Anal., vol. 58 (5), pp. 2632-2661, 2020.

\bibitem{WX2006} {\sc M. Wang and J. Xu}, {\em
The Morley element for fourth order elliptic equations in any dimensions}, Numer. Math., vol. 103, pp. 155-169, 2006.

\bibitem{WX2012} {\sc M. Wang and J. Xu}, {\em
Minimal finite element spaces for $2m$-th-order partial differential equations in $R^n$}, Math. Comp., vol. 82, pp. 25-43, 2012.

\bibitem{WX2007-2} {\sc M. Wang, Z. Shi and J. Xu}, {\em
Some n-rectangle nonforming elements for fourth order elliptic equations}, J. Comput. Math., vol. 25 (4), pp. 408-420, 2007.

\bibitem{WXH2006} {\sc M. Wang, J. Xu and Y. Hu}, {\em
Modified Morley element method for a fourth order elliptic singular perturbation problem}, J. Comput. Math., vol. 24 (2), pp. 113-120, 2006.

\bibitem{YZ2022} {\sc X. Ye and S. Zhang}, {\em
A weak divergence CDG method for the biharmonic equation on triangular and tetrahedral meshes}, Appl. Numer. Math., vol. 178, pp. 155-165, 2022.

\bibitem{YZ2020} {\sc X. Ye and S. Zhang}, {\em
A stabilizer free weak Galerkin method for the biharmonic equation on polytopal meshes}, SIAM J. Numer. Anal., vol. 58 (5), pp. 2572-2588, 2020.

\bibitem{ZXW2023} {\sc P. Zheng, S. Xie and X. Wang}, {\em
A stabilizer-free $C^0$ weak Galerkin method for the biharmonic equations}, Sci. China Math., vol. 66, pp. 627-646, 2023.


\end{thebibliography}
\end{document}